\documentclass[a4paper,11pt,reqno]{article}
\usepackage{amsmath,amssymb,hyperref,amsthm,enumerate,color,stmaryrd,pgf,tikz,comment,multirow}
\usetikzlibrary{positioning, arrows.meta, bending}
\usepackage[normalem]{ulem}
\usepackage{marginnote}
\usepackage{authblk}
\usepackage[left=2.6cm,right=2.6cm,top=3cm,bottom=3cm,bindingoffset=0cm]{geometry}
\title{Nowhere-zero $3$-flows in Cayley graphs on solvable groups of twice square-free order}
\author[1,2]{Milad Ahanjideh\thanks{This work was supported in part by the Slovenian Research and Innovation Agency (ARIS), research projects N1-0391 and J1-70035.}}
\author[1,2]{Istv\'an Kov\'acs\thanks{This work was supported in part by the Slovenian Research and Innovation Agency (ARIS), research program P1-0285 and reserach projects 
N1-0391, J1-50000, N1-0428, N1-0429, J1-70035, J1-70047, J1-70046 and 
N1-0481.}} 
\affil[1]{FAMNIT, University of Primorska, Galgolja\v{s}ka 8, SI-6000 Koper, Slovenia}
\affil[2]{IAM, University of Primorska, Muzejski trg 2, SI-6000 Koper, Slovenia}
\newtheorem{thm}{Theorem}[section]
\newtheorem{lem}[thm]{Lemma}
\newtheorem{prop}[thm]{Proposition}

\newtheorem{hypo}[thm]{Hypothesis}
\newtheorem{conj}[thm]{Conjecture}
\newtheorem*{claim}{Claim}

\theoremstyle{definition}
\newtheorem{exm}[thm]{Example}
\theoremstyle{remark}
\newtheorem{rem}[thm]{Remark}
\def\aut{\mathrm{Aut}}
\def\cay{\mathrm{Cay}}
\def\edge{\mathrm{E}}
\def\G{\Gamma}
\def\H{\mathcal{H}}
\def\Null{\mathrm{Null}}

\def\vertex{\mathrm{V}}
\def\Y{\mathcal{Y}}
\def\Z{\mathbb{Z}}
\newcommand{\sg}[1]{\langle {#1}\rangle}
\begin{document}
\maketitle
%---------------------------------------------------------------------------------------------------%
\begin{abstract}
We verify Tutte's $3$-flow conjecture in the class of Cayley graphs on solvable groups of 
order $2n$, where $n$ is square-free.  The proof relies on a new necessary and sufficient condition for a simple $5$-valent graph to admit a nowhere-zero 
$3$-flow in terms of a special pseudoforest decomposition.  \\ [+1ex]
[+1ex]
\noindent \emph{Math. Subj. Class.:} 05C21, 05C25.
\end{abstract}
%---------------------------------------------------------------------------------------------------%
\section{Introduction}\label{sec:intro}

Let $\G$ be a finite simple graph. 
An \emph{orientation} of $\G$ is a digraph $D$ whose vertex set is equal to the vertex set  
$\vertex(\G)$ and its arcs are obtained from $\G$ by endowing each edge of $\G$ with one of the two possible directions. For an edge $e$ of $\G$, let $D(e)$ denote the associated arc of $D$, and for a vertex $v \in \vertex(\G)$, denote by 
$D^+(v)$ the set of edges $\G$ that the associated arcs in $D$ have tail $v$ and by $D^-(v)$ the set of those edges that the associated arcs in $D$ have head $v$.  Let $A$ be an abelian group with the operation written additively and with identity element $0$. An \emph{$A$-flow} in 
$\G$ is a pair  $(D,\varphi)$, 
where $D$ is an orientation of $\G$ and $\varphi$ is an $A$-valued function on the edge 
set $\edge(\G)$, such that for every vertex $v \in  \vertex(\G)$,
\[
\sum_{e\in D^{+}(v)}\varphi(e)=\sum_{e\in D^{-}(v)}\varphi(e).
\]
In addition, if $\varphi(e) \ne 0$ for every $e \in \edge(\G)$, then $(D,\varphi)$ is called a \emph{nowhere-zero $A$-flow}.  A $\Z$-flow $(D,\varphi)$ is called a $k$-flow if 
$|f(e)|< k$ for every $e \in \edge(\G)$, here $\Z$ stands for the additive group 
of integers and $k$ is a fixed positive integer. 
The concept of nowhere-zero flows was introduced by Tutte~\cite{T49,T54} in connection with his work on the Four Color Conjecture. He showed that the Four Color Conjecture holds 
if and only if all bridgeless planar graphs admit a nowhere-zero $4$-flow.  
Tutte proposed three conjectures on $k$-flows in graphs which are still open. 
In this paper we are interested in his $3$-flow conjecture, which reads as follows.

\begin{conj}\label{conj}
{\rm (Tutte)} Every $4$-edge-connected graph admits a nowhere-zero $3$-flow. 
\end{conj}

This conjecture has been studied extensively. Jaeger~\cite{J} showed that every $4$-edge-connected graph admits a nowhere-zero $4$-flow and he also conjectured the existence of 
a positive integer $k$, for which every $k$-edge-connected graph admits a 
nowhere-zero $3$-flow.  Thomassen~\cite{T} showed that Jaeger's conjecture holds for 
$k=8$, which was improved by Lov\'asz et al.\,\cite{LTWZ}, who proved that 
every $6$-edge-connected graph admits a nowhere-zero $3$-flow. 

In this paper, we focus on Conjecture~\ref{conj} in the class of Cayley graphs.   
This class has attracted considerable attention~\cite{AI,LL,NS,PSS,YL,ZZ,ZhZ,ZsZ}. 
Let $G$ be a finite group with identity element $1$ and $X \subseteq G$ be a subset 
such that $1 \notin X$ and $X=X^{-1}:=\{x^{-1} \mid x \in X\}$. 
Then the \emph{Cayley graph} $\cay(G,X)$ is defined to have vertex set $G$ and for any $g, h \in G$, there is an edge joining $g$ and $h$ if and only if $g^{-1}h \in X$. The set $X$ is 
referred to as the \emph{connection set}.   
It is well-known that every connected $k$-valent Cayley graph is $k$-edge-connected~\cite{M}, and therefore, in the class of Cayley graphs, Conjecture~\ref{conj} is equivalent to saying that every connected Cayley graph on a group $G$ of valency at least $4$ admits a nowhere-zero $3$-flow. This statement is known to be true if $G$ is 
\begin{enumerate}[(i)]
\item an abelian~\cite{PSS} or nilpotent~\cite{NS} or dihedral group~\cite{YL}; or  
\item a generalized dihedral or generalized quaternion~\cite{LL}~or 
generalized dicyclic group~\cite{AI}; or 
\item a supersolvable group with a non-cyclic Sylow $2$-subgroup or a  
group with a derived subgroup of square-free order~\cite{ZsZ}; or  
\item a group of order $p^2q$~\cite{ZZ} or $8p$~\cite{ZhZ}, where $p$ and $q$ are distinct primes.    
\end{enumerate}
Note that all groups listed in (i)--(iv) are solvable. In this paper, we consider another infinite family of solvable groups.  Our main result is the following theorem.

\begin{thm}\label{main}
Let $G$ be a solvable group of order $2n$, where $n$ is square-free. 
Then every connected Cayley graph of valency at least $4$ on $G$ 
admits a nowhere-zero $3$-flow.
\end{thm}

\begin{rem}\label{rem:1}
It is worth mentioning that the result on the groups in (i)--(iii) was also shown to be true 
for the broader family of Cayley multigraphs (for the definition, see Section~\ref{sec:known}). 
However, this is not true for the groups considered in Theorem~\ref{main}. An 
example of a $5$-valent Cayley multigraph on the alternating group $A_4$, which 
does not admit a nowhere-zero $3$-flow, is provided in Example~\ref{A4}. 
\end{rem}

The paper is organized as follows. In Section~\ref{sec:known}, we collect the concepts and results needed throughout the paper. Previous works follow an approach based on 
detecting certain closed ladders as subgraphs. This approach was proposed by N{\'a}n{\'a}siov{\'a} and \v{S}koviera~\cite{NS} and was later developed to the concept of generalized closed ladders by Zhang and Zhou~\cite{ZsZ}. Using this method, we derive a new sufficient condition for a graph of valency $5$ to admit a nowhere-zero $3$-flow (see Lemma~\ref{new-suffi}). 

In Section~\ref{sec:pforests}, we introduce a new approach based on pseudoforests. 
Roughly speaking, we give a necessary and sufficient condition for a $5$-valent graph
 $\G$ to admit a nowhere-zero $3$-flow in terms of a partition of the vertex set 
 $\vertex(\G)$ into two subsets such that the subgraphs of $\G$ induced by 
these subsets are pseudoforests and there exists a matching of $\G$ between 
these subsets satisfying certain condition (see Lemma~\ref{new-condition}).   
After all preparations are completed, the proof of Theorem~\ref{main} is presented in Section~\ref{sec:proof}.   
%---------------------------------------------------------------------------------------------------%
\section{Preliminaries}\label{sec:known}

All groups and graphs in this paper are finite and unless otherwise stated, all graphs are simple. For a graph $\G$, the vertex set of $\G$ is denoted by $\vertex(\G)$ and 
the edge set of $\G$ by $\edge(\G)$.  If $\{u,v\}$ is an edge of $\G$, 
then we will write $u \sim v$. 
For subgraphs $\G_1$ and $\G_2$ of $\G$ with $\vertex(\G)=\vertex(\G_1) \cup \vertex(\G_2)$ and $\edge(\G)=\edge(\G_1) \cup \edge(\G_2)$, we write $\G=\G_1 \cup \G_2$. 
For a subset $U \subseteq \vertex(\G)$, the subgraph of $\G$ induced by $U$ is denoted by $\G[U]$. If $W \subseteq \vertex(\G)$ is another subset disjoint with $U$, then $\G[U,W]$ denotes the bipartite graph having bipartition parts $U$ and $W$ and the edges $\{u,w\} \in \edge(\G)$ with $u \in U$ and $w \in W$. For a subset $E \subseteq \edge(\G)$, the subgraph of $\G$ obtained from $\G$ by deleting all edges in $E$ from $\edge(\G)$ is denoted by $\G - E$. 
If each vertex of $\G$ has even valency, then $\G$ is called \emph{even}; and a subgraph 
$\G'$ of $\G$ is called a \emph{parity subgraph} if $\G-\edge(\G')$ is an 
even graph. 

\begin{prop}[{\cite[Lemma~2.4]{ZZ}}]\label{parity}
Let $\G$ be a graph and $k$ be a positive integer greater than $2$. If $\G$ has a parity 
subgraph admitting a nowhere-zero $k$-flow, then $\G$ admits a nowhere-zero 
$k$-flow. 
\end{prop}

Let $G$ be a group with identity element $1$. A multiset $X \subseteq G$ is called a \emph{connection multiset} if $1 \notin X$, $X=X^{-1}=\{x^{-1} \mid x \in X\}$, and for every $x \in X$, the multiplicity of $x$ and $x^{-1}$, respectively, are equal. 
Then the \emph{Cayley multigraph} $\cay(G,X)$ is defined to have vertex set $G$ and for any $g, h \in G$, there are exactly $\ell$ edges joining $g$ and $h$ if and only if $g^{-1}h \in X$ with multiplicity $\ell$ in $X$. $A$-flows in Cayley multigraphs (more generally, in finite loopless graphs with parallel edges allowed) are defined in the same way as in 
simple graphs. 

Let $\G=\cay(G,X)$ be a Cayley graph or multigraph on a group $G$. For $x \in X$, we refer to the edges of the form $\{g,gx\}$, where $g$ runs over $G$, as the \emph{$x$-edges}. 
Note that the set of $x$-edges of $\G$ coincides with the set of its $x^{-1}$-edges.     
Suppose that $N$ is a normal subgroup of $G$ with $X \cap N=\emptyset$. 
Then the multiset $X/N \subseteq G/N$, defined as $X/N=\{ Nx \mid x \in X\}$, is a 
either an inverse closed simple subset or a 
connection multiset of the quotient group $G/N$. The Cayley graph or multigraph $\cay(G/N,X/N)$ is 
called the \emph{quotient} of $\G$ with respect to $N$, denoted by $\G/N$. 

\begin{prop}[{\rm \cite[Proposition~4.1]{NS}}]\label{quo}
Let $G$ be a group and $N$ be a normal subgroup of $G$. 
Let $X$ be a connection multiset of $G$ such that $N \cap X=\emptyset$. 
If $\cay(G/N,X/N)$ admits a nowhere-zero $k$-flow, then so does $\cay(G,X)$.
\end{prop}

Let $n \ge 2$ be an integer. The \emph{circular ladder} $CL_n$ is the Cayley graph 
$\cay(\Z_n \times \Z_2,\{(1,0),(n-1,0),(0,1)\})$, where $\Z_n$ denotes the 
additive group of integers modulo $n$.  The $(0,1)$-edges of $CL_n$ are called the \emph{rungs} and the remaining edges are called the \emph{rail edges}.  The \emph{M\"obius ladder} $ML_n$ is the Cayley graph 
$\cay(\Z_{2n},\{1,2n-1,n\})$. The $n$-edges of $ML_n$ are called the \emph{rungs} 
and the remaining edges are called the \emph{rail edges}. 
By a \emph{closed ladder} we mean a circular or M\"obius ladder.  
Part (i) of the following lemma is \cite[Lemma~3.1(ii)]{ZsZ}; 
part (ii) follows from the proof of \cite[Lemma~3.2]{NS}.

\begin{prop}\label{CLn}
Let $n \ge  2$ be an integer. Then the following hold.
\begin{enumerate}[(i)]
\item $CL_n$ admits a nowhere-zero $3$-flow if and only if $n$ is even. 
\item If $n$ is odd and $e$ is any rung of $CL_n$, then $CL_n$ admits a $3$-flow 
$(D,\varphi)$ such that $\varphi$ takes $0$ exactly on $e$. 
\end{enumerate}
\end{prop}

The method using closed ladders in the investigation of Tutte's conjecture in the 
class of Cayley graphs was initiated in~\cite{NS}, where the following results were proved. 

\begin{prop}[{\cite[Theorem~3.3]{NS}}]\label{c-involution}
Let $\cay(G,X)$ be a Cayley multigraph of valency at least $4$ such that 
$X$ contains a central involution of $G$. Then $\cay(G,X)$ admits a nowhere-zero 
$3$-flow. 
\end{prop}

\begin{prop}[{\cite[Theorem~4.3]{NS}}]\label{nilpotent}
Every Cayley graph of valency at least $4$ on a nilpotent group admits a nowhere-zero 
$3$-flow.
\end{prop}

The method of closed ladders was developed by Zhang and Zhou~\cite{ZsZ}, who proved the following 
result. 

\begin{prop}[{\cite[Theorem~1.3]{ZsZ}}]\label{derived}
Every Cayley graph of valency at least $4$ on a group whose derived subgroup 
is of square-free order admits a nowhere-zero $3$-flow.
\end{prop}

Regarding terminology and notation in group theory, we follow the book~\cite{Hbook}. In particular, if $x, y \in G$, then $|x|$ denotes 
the \emph{order} of $x$, and $x^y$ is written for the \emph{conjugate} $y^{-1}xy$. 
The \emph{semidirect product} of $G$ with another group $H$ is denoted by $G \rtimes H$. 
The \emph{center}, \emph{derived subgroup} and \emph{Fitting subgroup} of $G$ are denoted by $Z(G), G^\prime$ and $F(G)$, respectively.  
For a subgroup $H \le G$, the \emph{centralizer} of 
$H$ in $G$ is denoted by $C_G(H)$. We shall use the following facts.

\begin{prop}[{\cite[Theorem~1.4.5]{Hbook}}]\label{N/C}
If $N$ is a normal subgroup of $G$, then $G/C_G(N)$ is isomorphic to a subgroup of 
$\aut(N)$.
\end{prop}

\begin{prop}[{\cite[Theorem~3.4.2(b)]{Hbook}}]\label{Fit}
If $G$ is a solvable group, then $C_G(F(G)) \le F(G)$. 
\end{prop}

We conclude the section with a lemma that provides a sufficient condition for a $5$-valent simple graph to admit a nowhere-zero $\Z_3$-flow. Tutte showed that the existence of $\Z_k$-flow in a graph $\G$ is equivalent to the existence of $k$-flow (see, e.\,g.~\cite[Theorem~21.3]{BMbook}). 
In this paper, we work with $\Z_3$-flows rather than $3$-flows. 
The proof of the lemma illustrates the use of 
circular ladders in the study of $3$-flows in graphs. Before stating the lemma, we 
introduce some additional notation. 

Let $(D,\varphi)$ be an $A$-flow  
in a graph $\G$, where $A$ is an abelian group with identity element $0$. 
The \emph{reverse orientation} $D^t$ of $\G$ obtained from $\G$ by 
reversing each arc of $D$. 
The function $-\varphi : \edge(\G) \to A$ is defined as 
$(-\varphi)(e)=-\varphi(e)$ for every $e\in \edge(\G)$. Clearly, 
both pairs $(D^t,\varphi)$ and $(D,-\varphi)$ are $A$-flows. 
The subset $\Null(\varphi) \subseteq \edge(\G)$ is defined as 
\[
\Null(\varphi)=\{ e \in \edge(\G) \mid \varphi(e)=0\}.
\] 
Finally, for integers $a < b$, throughout the paper we use the symbol $[a,b]$ to denote 
the set $\{a,a+1,\ldots,b\}$.

\begin{lem}\label{new-suffi}
Let $\G$ be a simple graph of order $n$ and valency $5$ and suppose that it contains  subgraphs $\Sigma$ and $\Lambda$ satisfying the following conditions.
\begin{enumerate}[(i)]
% (i)
\item $\Sigma$ is a spanning subgraph of $\G$ with components 
$\Sigma_1,\ldots,\Sigma_m$, and  there exists an odd integer $t > 1$ such that 
$\Sigma_i \cong CL_t$ for each $i \in [1,m]$ (hence $n=2tm$).
% (ii)
\item $\Sigma \cup \Lambda$ is a parity subgraph of $\G$.  
% (iii)
\item For each $i \in [1,m]$, 
$|\edge(\Lambda) \cap \edge(\Sigma_i)| \in \{1,2\}$ and 
every edge in $\edge(\Lambda) \cap \edge(\Sigma_i)$ is a rung of $\Sigma_i$.
% (iv)
\item $\Lambda$ admits a $\Z_3$-flow $(D,\varphi)$ such that 
\begin{enumerate}[(a)]
\item $\Null(\varphi) \subseteq \edge(\Lambda) \cap \edge(\Sigma)$. 
\item For each $i \in [1,m]$, $\edge(\Lambda) \cap \edge(\Sigma_i) \not\subset \Null(\varphi)$.
\end{enumerate}
\end{enumerate}
Then $\G$ admits a nowhere-zero $\Z_3$-flow. 
\end{lem}
\begin{proof}
Our goal is to construct a nowhere-zero $\Z_3$-flow $(D',\varphi')$ in 
$\Sigma \cup \Lambda$. 
The lemma will then follow from (ii) and Proposition~\ref{parity}. 

To achieve this goal, we define first a $\Z_3$-flow $(D_i,\varphi_i)$ of the component 
$\Sigma_i$ for every $i \in [1,m]$. 
According to case (b) in (iv), for every $i \in [1,m]$, there is a common edge 
$e_i$ of $\Sigma_i$ and $\Lambda$ for which $\varphi(e_i) \ne 0$. 
In view of (iii), we may assume without loss of generality that 
the remaining common edges of 
$\Sigma$ and $\Lambda$, if there is any, can be listed as 
$e'_1,\ldots,e'_l$, where $l \in [1,m]$ and $e'_i$ lies in the component $\Sigma_i$ for every 
$i \in [1,l]$. 

Fix $i \in [1,m]$. Since $\Sigma_i \cong CL_t$ and $t$ is odd, Proposition~\ref{CLn} guarantees the existence of a $\Z_3$-flow $(D_i,\varphi_i)$ in $\Sigma_i$ with 
$\Null(\varphi)=\{e_i\}$. If $i \le l$ and $D_i(e'_i) \ne D(e'_i)$, then replace 
$D_i$ with $D_i^t$.  Then, as $\varphi_i(e_i)=0$, we may choose the orientation of the edge 
$e_i$ in $D_i$ so that $D_i(e_i)=D(e_i)$. If $i \le l$ and $\varphi(e'_i)+\varphi_i(e'_i)=0$, then 
replace  $\varphi_i$ with $-\varphi_i$. As $\varphi_i(e'_i) \ne 0$, $\varphi(e'_i)-\varphi_i(e'_i) \ne 0$. As a result, we have constructed $\Z_3$-flows $(D_i,\varphi_i)$ in $\Sigma_i$, $i \in [1,m]$, which have the following properties:
\begin{align}
\Null(\varphi_i)=\{e_i\} &~~(i \in [1,m]), &  \label{eq:null} \\  
D_i(e_i)=D(e_i),~D_j(e'_j)=D(e'_j) & ~~(i \in [1,m], j \in [1,l]), \label{eq:Di} \\ 
\varphi_i(e'_i)+\varphi(e'_i) \ne 0 &~~(i \in [1,l]). \label{eq:varphii}
\end{align}

Observe that the orientations $D$ and $D_i$ coincide on the common edges of $\Lambda$ and $\Sigma_i$ for every $i \in [1,m]$. Consequently, there is a uniquely defined orientation $D'$ of $\Sigma \cup \Lambda$ whose restriction to   
$\Lambda$ coincides with $D$, and for every $i \in [1,m]$, the restriction of $D'$ 
to $\Sigma_i$ coincides with $D_i$. 
Now, define the function $\varphi' : \edge(\Sigma \cup \Lambda) \to \Z_3$ as 
\[
\varphi'(e)=\begin{cases}
\varphi(e) & \text{if}\quad e \in E(\Lambda) \setminus E(\Sigma), \\
\varphi_i(e) & \text{if}\quad e \in E(\Sigma_i) \setminus E(\Lambda), \\
\varphi(e)+\varphi_i(e) & \text{if}\quad e\in E(\Lambda)\cap E(\Sigma_i) . 
\end{cases}
\]    
It follows from the assumption in case (iv) of the lemma and the properties in \eqref{eq:null}--\eqref{eq:varphii} that $(D',\varphi')$ is  a nowhere-zero $\Z_3$-flow.	
\end{proof}
%---------------------------------------------------------------------------------------------------%
\section{Nowhere-zero $3$-flows and a pseudoforest decomposition}\label{sec:pforests} 

Let $\G$ be a graph and $D$ be an orientation of $\G$. For a vertex $v \in \vertex(\G)$, 
the \emph{out-degree} $d^+_D(v)$ of $v$ in $D$ is equal to 
the cardinality $|D^+(v)|$ and 
the \emph{in-degree} $d^-_D(v)$ is equal to the cardinality 
$|D^-(v)|$.  In the case where every vertex has out-degree $0$ or $1$, we say that $D$ is a \emph{$(0,1)$-orientation}. In the next lemma we describe the graphs that admit at least one $(0,1)$-orientation. For this purpose, we need to mention two more definitions. 
A graph is called \emph{unicyclic} if it contains exactly one cycle and it is called a  
\emph{pseudoforest} if each of its components is a tree or unicyclic. 

\begin{lem}\label{pforest}\
\begin{enumerate}[(i)] 
\item Let $D$ be a $(0,1)$-orientation of a graph $\G$. Then $\G$ is a pseudoforest. 
\item Let $\G$ be a pseudoforest and fix a subset $V' \subseteq  \vertex(\G)$. 
Then there exists a $(0,1)$-orientation $D$ of $\G$ such that 
\[
V'=\{ v \in \vertex(\G) \mid d^+_D(v)=0\}
\]
if and only if $V'$ has exactly one vertex from each tree component of $\G$ and no vertex  
from the unicyclic components.
\end{enumerate}
\end{lem}
\begin{proof}
(i): Let $U=\{ v \in \vertex(\G) \mid d^+_D(v)=0\}$ and $\G'$ be a component of $\G$. 
Then the following inequality holds for $\G'$:  
\[
|\vertex(\G')|-1 \le |\edge(\G')|=\sum_{v \in \vertex(\G')} d^+_D(v)=|\vertex(\G')|-
|U \cap \vertex(\G')|.
\]
Thus $|U \cap \vertex(\G')| \le 1$. Furthermore, if $|U \cap \vertex(\G')|=1$, then 
$|\edge(\G')|=|\vertex(\G')|-1$, whence $\G'$ is a tree; whereas if $|U \cap \vertex(\G')|=0$, then  $|\edge(\G')|=|\vertex(\G')|$, whence $\G'$ is unicyclic. 
\medskip

(ii): The necessity of the statement can be red off from the proof of (i). For the sufficiency 
assume that $V'$ has exactly one vertex from each tree component and none from the 
unicyclic components.  A required $(0,1)$-orientation can be constructed as follows. 
If $\G'$ is a tree component of $\G$, then orient each edge of $\G'$ toward its unique 
vertex in $V'$; whereas if $\G'$ is a unicyclic component, then orient the edges in 
the cycle of $\G'$ cyclically, and then 
the remaining edges toward the cycle. 
\end{proof}

In what follows, we say that a set $U$ of vertices of a pseudoforest $\G$ is a \emph{transversal} of $\G$ if $U$ has exactly one vertex from each tree component of $\G$ and no vertex from the unicyclic components. 
\medskip

We are now ready to state the main result of this section. This will serve as a key  
tool in the proof of Theorem~\ref{main}.

\begin{lem}\label{new-condition}
Let $\G$ be a graph of valency $5$. Then $\G$ admits a nowhere-zero $\Z_3$-flow if and only if 
$\vertex(\G)$ can be partitioned into non-empty subsets $U$ and $W$ such that the following conditions hold.
\begin{enumerate}[(i)]
\item The induced subgraphs $\G[U]$ and $\G[W]$ are pseudoforests.
\item Either all components of $\G[U]$ and $\G[W]$ are unicyclic, or there 
exist a transversal $U'$ of $\G[U]$ and a transversal $W'$ of $\G[W]$ such that 
$\G[U',W']$ has a perfect matching.
\end{enumerate}
\end{lem}    
\begin{proof}
Part ``$\Rightarrow$'': Assume that $(D,\varphi)$ is a nowhere-zero $\Z_3$-flow of $\G$. 
We may assume without loss of generality that 
$\varphi$ assigns the same non-zero value to every edge of $\G$. 
Choose an arbitrary vertex $v \in \vertex(\G)$ and set 
$k=d^+_D(v)$. As $\varphi$ is a constant function, we have that $k \equiv 5-k \pmod 3$, 
from which $k=1$ or $k=4$. Define the subsets $U, W \subseteq \vertex(\G)$ as
\[
U=\{ v \in \vertex(\G) \mid d^+_D(v)=1 \},~
W=\{ v \in \vertex(\G) \mid d^+_D(v)=4 \}.
\]

Assume that $W=\emptyset$. Then it follows from 
Lemma~\ref{pforest} that $\G$ is a pseudoforest whose components are unicyclic. 
Thus, if $\G'$ is such a component, then $|\edge(\G')|=|\vertex(\G')|$. 
However, as $\G'$ is regular of valency $5$, we also have 
$2|\edge(\G')|=5|\vertex(\G')|$. This is impossible, hence $W \ne \emptyset$. 
Applying the same argument to the reverse orientation $D^t$ shows that 
$U \ne \emptyset$ as well. 
We show that the conditions in (i) and (ii) hold for $\G[U]$ and $\G[W]$. 

Denote by $D_1$ the restriction of $D$ to $\G[U]$ and 
by $D_2$ the restriction of $D^t$ to $\G[W]$. 
Then, for $v \in U$, $d^+_{D_1}(v) \le d^+_D(v)=1$; and for $v \in W$, 
\[
d^+_{D_2}(v) \le d^+_{D^t}(v)=5-d^+_D(v)=1.
\]
By Lemma~\ref{pforest}(i), both $\G[U]$ and 
$\G[W]$ are pseudoforests, so the condition in (i) holds.

Define the subsets $U' \subseteq U$ and $W' \subseteq W$ as  
\[
U'=\{v \in U \mid d^+_{D_1}(v)=0 \},~
W'=\{v \in W \mid d^+_{D_2}(v)=0 \}.
\]
It follows from Lemma~\ref{pforest}(ii) that $U'$ is a transversal of 
$\G[U]$ and $W'$ is a transversal of $\G[W]$.

Fix an arbitrary vertex $u \in U'$. Then there is a unique neighbor 
$w$ of $u$ such that $(u,w)$ is an arc of $D$ and $w \in W$. 
It follows that $w \in W'$ and this allows us to define the mapping $f : U' \to W'$ by 
letting $f(u)=w$. To have the condition in (ii), it is sufficient to show that $f$ is bijective. 
If $w' \in W'$, then there is a unique neighbor $u'$ of $w'$ such 
that $(u',w')$ is an arc of $D$ and $u' \in U$. 
It follows that $u' \in U'$, hence $f(u')=w'$, and so $f$ is surjective.  
If $u, u' \in U'$ were distinct vertices such that $f(u)=f(u')$, then we would have 
$d^+_D(f(u)) < 4$. This is impossible, hence $f$ is also injective. 
\medskip

Part ``$\Leftarrow$'': Assume that $\vertex(\G)$ is partitioned into non-empty subsets $U$ and $W$ such that the conditions in (i) and (ii) hold for $\G[U]$ and $\G[W]$. In particular, denote by $M$ a  perfect matching in $\G[U',W']$. By Lemma~\ref{pforest}, there exist a $(0,1)$-orientation $D_U$ of $\G[U]$ and a $(0,1)$-orientation $D_W$ of $\G[W]$ such that 
\[
U'=\{v \in U \mid d_{D_U}^+(v)=0\},~W'=\{v \in W \mid d_{D_W}^+(v)=0\}.
\]
Then, define the orientation $D$ of $\G$ by orienting the edges of $\G[U]$ as in $D_U$, the edges of $\G[W]$ as in $(D_W)^t$, the edges in $M$ from $U'$ to $W'$, and finally, the edges 
in $\G[U,W]$ outside $M$ from $W$ to $U$. Then for every $v \in \vertex(\G)$, 
\[
d^+_D(v)=
\begin{cases} 1 & \text{if}~v \in U, \\ 
4 & \text{if}~v \in W.
\end{cases}
\]
Clearly, the pair $(D,\varphi)$ defines a nowhere-zero $\Z_3$-flow in 
$\G$, where $\varphi$ assigns a constant non-zero value to every edge of $\G$. 
\end{proof}  
%---------------------------------------------------------------------------------------------------%
\section{Proof of Theorem~\ref{main}}\label{sec:proof}

Throughout this section, we assume that $n$ is a square-free integer and $G$ is 
a solvable group of order $2n$. Our goal is to show that every connected Cayley graph on $G$ of valency at least $4$ admits a nowhere-zero $\Z_3$-flow.  
We proceed by induction on $n$. If $n$ is a prime, then $G$ is either an abelian or a dihedral group and the assertion was proved in~\cite{PSS,YL}. 
The following assumption on $n$ will be used throughout this section. 

\begin{hypo}\label{n}
The number $n$ is composite and  
for every solvable group $A$ of order $2n'$, where $n'$ is a proper divisor of $n$, 
every connected Cayley graph on $A$ of valency at least $4$ admits a nowhere-zero 
$\Z_3$-flow.
\end{hypo}

Suppose that $\G=\cay(G,X)$ is a connected Cayley graph of valency $d \geq 4$. 
If $d=4$, then $\G$ is Eulerian and hence admits a 
nowhere-zero $\Z_3$-flow in a straightforward manner. 
If $d > 5$, 
then $\G$ is $6$-edge-connected~\cite{M} and 
therefore, it admits a nowhere-zero $\Z_3$-flow~\cite{LTWZ}. 
We prove through a sequence of lemmas that 
$\G$ admits a nowhere-zero $\Z_3$-flow also in the case where $d=5$. 

We begin with two observations concerning nowhere-zero $\Z_3$-flows in 
Cayley multigraphs. The first one is included for curiosity, as it demonstrates that Theorem~\ref{main} cannot be extended to multigraphs (see Remark~\ref{rem:1}). 
The second one is an auxiliary lemma that will be needed later. 

\begin{exm}\label{A4}
Let $\G=\cay(A_4,\{a,b,b,b^{-1},b^{-1}\})$ be the multigraph, where $A_4$ is the alternating 
group of degree $4$, $A_4=\sg{a,b}$, $|a|=2$ and $|b|=3$. 
We claim that there does not exist a nowhere-zero $\Z_3$-flow in $\G$.
\begin{proof}
Toward a contradiction suppose that $(D,\varphi)$ is a nowhere-zero $\Z_3$-flow in $\G$. 
We may assume without loss of generality that parallel edges have the same orientation in $D$.  

Define the spanning subgraph $\G'$ of $\G$ and parallel the $\Z_3$-flow $(D',\varphi')$ in $\G'$ 
as follows. 
\begin{itemize}
\item If $e$ is an $a$-edge of $\G$, then add $e$ to $\G'$ and set $D'(e)=D(e)$,  
$\varphi'(e)=\varphi(e)$; 
\item if $e, e'$ are parallel $b$-edges of $\G$ and   
$\varphi(e) \ne \varphi(e')$, then add none of $e$ and $e'$ to $\G'$; 
\item if $e, e'$ are parallel $b$-edges of $\G$ and $\varphi(e)=\varphi(e')$, then add 
only $e$ to $\G'$ and set $D'(e)=D(e)$, $\varphi'(e)=-\varphi(e)$. 
\end{itemize}
It is not difficult to see that $\G'$ is a simple graph, which is obtained from the 
Cayley graph $\Delta=\cay(A_4,\{a,b,b^{-1}\})$ by deleting at most one edge from each 
of the four triangles of $\Delta$. Clearly, $(D',\varphi')$ is a nowhere-zero $3$-flow in 
$\G'$.  Now, smoothing the vertices of degree $2$ in $\G'$ yields a $3$-valent non-bipartite graph, which also admits a nowhere-zero $\Z_3$-flow. However, this contradicts the well-known fact that a $3$-valent graph admits a nowhere-zero $\Z_3$-flow if and only if it is bipartite 
(see~\cite[Proposition~6.4.2]{Dbook}).
\end{proof}
\end{exm}

\begin{lem}\label{multi}
Let $A$ be any group and $Y$ be connection multiset of $A$ of cardinality $5$ (computed 
with multiplicity) such that $Y$ contains an element of order larger than $2$ and of multiplicity $1$. Then $\cay(A,Y)$ admits a nowhere-zero $\Z_3$-flow. 
\end{lem}
\begin{proof}
It is easy to see that $Y$ can be written of the form 
$Y=\{y,y^{-1},z,z,z'\}$, $|y| \ne 2$ and $|z|=|z'|=2$. 
Let $\Sigma=\cay(A,Y)$. 

Assume first that $z=z'$. 
The $y$-edges in $\Sigma$ induce $|y|$-cycles and the 
$z$-edges induce  parallel edges of multiplicity $3$. 
A nowhere-zero $\Z_3$-flow in $\Sigma$ can be constructed by 
orienting the $|y|$-cycles cyclically, the parallel edges identically, and then 
assigning the same non-zero value to every edge. 

Now, assume that $z \ne z'$. Note that there exists a nowhere-zero 
$\Z_3$-flow $(D,\varphi)$ in the subgraph $\cay(A,\{y,y^{-1},z,z'\})$. This 
extends to a nowhere-zero $\Z_3$-flow $(D',\varphi')$ in $\Sigma$ as follows. 
If $e$ is a $y$- or a $z'$-edge of $\Sigma$, then set $D'(e)=D(e)$ and 
$\varphi'(e)=\varphi(e)$. If $e, e'$ are parallel $z$-edges of $\Sigma$, then set 
$D'(e)=D'(e')=D(e)$ and $\varphi'(e)=\varphi'(e')=-\varphi(e)$. 
\end{proof}

We begin our analysis of the connected $5$-valent Cayley graphs on the group 
$G$ defined at the beginning of the section. In the following lemma, we reduce the question of 
whether such a graph admits a nowhere-zero $\Z_3$-flow to four specific graphs. 

\begin{lem}\label{reduction}
Let $\G=\cay(G,X)$ be a $5$-valent connected Cayley graph. 
Then $\G$ admits a nowhere-zero $\Z_3$-flow, unless one of the following holds. 
\begin{enumerate}[(i)]
\item $G \cong (\Z_2^2 \times \Z_p) \rtimes \Z_{3k}$ and $\G$ is equal to 
\[
\G_1=\cay(G,\{x,a,a^{-1},y,y^{-1}\})~\text{or}~\G_2=\cay(G,\{x,ay,(ay)^{-1},y,y^{-1}\}),
\] 
where $p$ is a prime, $p > 3$, $|a|=p$, $|x|=2$, $|y|=3k$, $k$ is odd and not divisible by $p$    
and $xy \ne yx$. %, and $C_G(\sg{a,x})=\sg{a,x}$. 
\item $G \cong  A_4 \times \Z_p$ and $\G$ is equal to 
\[
\G_3=\cay(G,\{x,a,a^{-1},y,y^{-1}\})~\text{or}~
\G_4=\cay(G,\{x,ay,(ay)^{-1},y,y^{-1}\}),
\] 
where $p$ is a prime, $p > 3$, $|a|=p$, $|x|=2$ and $|y|=3p$. 
\end{enumerate}
\end{lem}
\begin{proof}
Assume that $\G$ does not admit a nowhere-zero $\Z_3$-flow. We show that (i) or (ii) holds 
for $G$ and $\G$. 
It follows from Proposition~\ref{derived} that $n$ is divisible by $2$. 
For the sake of an easier notation, we write $F=F(G)$.

Assume for the moment $|F|$ is not 
divisible by $4$. Then $F$ is a cyclic group and it follows from 
Proposition~\ref{Fit} that $F=C_G(F)$. By proposition~\ref{N/C}, 
$G/F$ is isomorphic to a subgroup of $\aut(F)$, and as $F$ is cyclic, 
$\aut(F)$ is abelian. However, then $G^\prime \le F$, hence  
$\G$ admits a nowhere-zero $\Z_3$-flow due to Proposition~\ref{derived}, a contradiction.

Thus $4$ divides $|F|$, and so $G$ has a normal Sylow $2$-subgroup of order $4$. 
Let this subgroup be denoted by $S$. Note that all involutions of $G$ are contained in $S$, 
in particular, $S \cap X \ne\emptyset$. Fix $x \in S \cap X$.

It follows from Proposition~\ref{c-involution} and the assumption that there does not exist 
a nowhere-zero $\Z_3$-flow in $\G$ that $C_G(S) < G$. 
By Proposition~\ref{N/C}, $G/C_G(S)$ is isomorphic to a subgroup of $\aut(S)$. 
Using also that $|G/C_G(S)|$ is odd, we obtain that $S \cong \Z_2^2$ and 
$|G/C_G(S)|=3$. As $\sg{X}=G$, $X \setminus C_G(S) \ne \emptyset$. Fix 
$y \in X \setminus C_G(S)$. 

We claim that $|y|$ is odd. Indeed, if $|y|$ is even, then $y$ centralizes an element in 
$S$. This and the fact that $|y^2|$ is odd yield that $y$ centralizes the whole group $S$, contradicting that $y \notin C_G(S)$.

Now, if $S=F$, then $C_G(S)=S$, and we find that $G \cong A_4$. However, all connected Cayley graphs on the group $A_4$ of valency $5$ admit a nowhere-zero $\Z_3$-flow by~\cite{ZZ}, a contradiction. Thus, $S < F$. 
Let $P$ be a Sylow $p$-subgroup of $F$ for a prime $p > 2$. 
 
\begin{claim}
If $X \cap P=\emptyset$, then $X=\{x,y,y^{-1},ay,(ay)^{-1}\}$ for some $a \in P, a\ne 1$, and  
$xy \ne yx$.
\end{claim}
\begin{proof}[Proof of Claim] 
Let $\G'=\cay(G/P,X/P)$. 
Assume for the moment that $Py$ has multiplicity $1$ in the set $X/P$. 
Then $\G'$ admits a nowhere-zero $\Z_3$-flow. 
Indeed, if $\G'$ is a simple graph, then this follows directly from the assumption in 
Hypothesis~\ref{n}; while if $\G'$ is a multigraph, then this follows from Lemma~\ref{multi} 
because $Py$ is of order larger than $2$ and of multiplicity $1$. 
Now, by Proposition~\ref{quo}, $\G$ admits a nowhere-zero $\Z_3$-flow, a contradiction. 
Thus, $Py$ has multiplicity larger than $1$ in $X/P$. This means that there exists $a \in P, a \ne 1$, 
such that $ay \in X$. If $|ay|=2$, then $ay \in S$, so $y \in F \le C_G(S)$, a contradiction. 
Thus $(ay)^{-1} \ne ay$, and $X=\{x,y,y^{-1},ay,(ay)^{-1}\}$. Finally, 
$xy \ne yx$, for otherwise, $x$ would be in $Z(G)$, which is impossible. 
\end{proof}

We show next that $F=\sg{P,S}=P \times S$. Assume the contrary and let 
$Q$ be a Sylow $q$-subgroup of $F$ for a prime $q$, such that $q \ne 2$ and $q \ne p$. 
If $X \cap P \ne \emptyset$ and $X \cap Q \ne \emptyset$, then 
$X \subset F$, by which $G=F$. This cannot occur due to Proposition~\ref{nilpotent}, 
thus $X \cap P=\emptyset$ or 
$X \cap Q=\emptyset$. We may assume without loss of generality 
that $X \cap P=\emptyset$. 
By the Claim, $X=\{x,y,y^{-1},ay,(ay)^{-1}\}$ for some $a \in P, a \ne 1$. 
As $y \notin C_G(S)$ and $Q < C_G(S)$, $X \cap Q=\emptyset$ also holds. The proof of the Claim can be applied to $Q$, and we obtain that 
\[
X/Q=\{Qx, Qy, Qy^{-1}, Qay, Q(ay)^{-1}\}
\]
is a multiset. Now, using also that $|y|$ is odd, $Q$ is normal in $G$ and $\sg{Q,a,x} < F$, 
it is a routine exercise to derive that $y \in F$, which is a contradiction.

To sum up, we have shown that $F(G)=P \times S$, and there exists a generator $a$ of $P$ such that 
\[ 
X=\{x, a,a^{-1},y, y^{-1}\}~or~X=\{x,ay,(ay)^{-1},y,y^{-1}\},
\]
where $|x|=2$, $|y|$ is odd and $xy \ne yx$. It follows that $3$ divides both $|y|$ and 
$|G/F|$, by which $|y|=3k$ for some odd number $k$ and $p > 3$. 

Now, if $p$ does not divide $k$, then $G=F \rtimes \sg{y} \cong 
(\Z_2^2 \times \Z_p) \rtimes \Z_{3k}$, and so case (i) follows.
Finally, if $p$ divides $k$, then $ay=ya$. Then 
$y^3 \in C_G(F)=F$, implying in turn that $k=|y^3|=p$,   
$G=A_4 \times \Z_p$, and so case (ii) follows. 
\end{proof}

In order to derive Theorem~\ref{main}, it remains to show that there are 
nowhere-zero $\Z_3$-flows in the graphs $\G_i$'s defined in Lemma~\ref{reduction}. 
For the remainder of the paper $G$ is a group described in Lemma~\ref{reduction}.
We may assume without loss of generality that for the normal Sylow $2$-subgroup $S$ of $G$, 
$S=\{x_0,x_1,x_2,x_3\}$, $x_0=1$, $x_1=x$ and 
\begin{equation}\label{eq:y-on-S}
y^{-1}x_iy=x_{i+1}\quad (i \in [1,3]),
\end{equation} 
where $i+1$ is computed modulo $3$.
There exists a number $r \in [1,p-1]$ such that  $y$ acts on the subgroup $\sg{a}$ by conjugation as 
\begin{equation}\label{eq:y-on-ai}
y^{-1}a^iy=a^{ri}\quad (i \in [0,p-1]).
\end{equation}  
Define the subgroup $H \le G$ and the subsets $X_0, X_1 \subset G$ as 
\[
H=\sg{a,y^3},~X_0=\{1,x_3,x_1y,x_2y,y^2,x_3y^2\},~
X_1=\{x_1,x_2,y,x_3y,x_1y^2,x_2y^2\}.
\]
Note that $C_G(S)=S \times H$. Thus, 
$H$ is normal in $G$, $|H|=|G|/12$, and   
$X_0 \cup X_1$ is a transversal of $H$ in $G$. The following observation will be 
useful later. 

\begin{lem}\label{iso}
Let $G$ be the group and $\G$ be the graph described in (i) or (ii) of Lemma~\ref{reduction}.
Let $\H \subset H$ be a non-empty proper subset and  
$U=X_0\H \cup X_1(H \setminus \H)$. Then $\G[U] \cong \G[G \setminus U]$. 
\end{lem}
\begin{proof}
We have $|U|=6|H|=|G|/2$ and $x_1U \cap U = \emptyset$. 
It follows from these that $G \setminus U=x_1U$ and therefore, the automorphism of $\G$ 
sending $g$ to $x_1g$ ($g \in G$) induces an isomorphism from 
$\G[U]$ to $\G[G \setminus U]$. 
\end{proof}

We deal with the Cayley graphs $\G_1$ and $\G_2$ in next subsection. The other two Cayley graphs 
$\G_3$ and $\G_4$ will be discussed in Subsection~\ref{sec:4.2}. 
%---------------------------------------------------------------------------------------------------%
\subsection{Graphs $\Gamma_1$ and $\Gamma_2$}\label{sec:4.1}

\begin{lem}\label{G1}
$\G_1$ admits a nowhere-zero $\Z_3$-flow. 
\end{lem}
\begin{proof}
Recall that $G \cong (\Z_2^2 \times \Z_p) \rtimes \Z_{3k}$ and 
$\G_1=\cay(G,\{x,a,a^{-1},y,y^{-1}\})$. Let $K=\sg{x, a}$ and $L=\sg{x,y}$. 
Then $K \cong \Z_{2p}$ and $L \cong \Z_2^2 \rtimes \Z_{3k}$.   
Define the subgraphs $\Sigma, \Lambda$ of $\G$ as
\begin{align*}
\Sigma =& \cay(G,\{x,a,a^{-1}\}), \\ 
\Lambda =& \cay(G,\{x,y,y^{-1}\})[L \cup aL].
\end{align*}
We settle the lemma by showing that the conditions in (i)--(iv) 
of Lemma~\ref{new-suffi} hold for $\Sigma$ and $\Lambda$. 
\medskip

\noindent $\bullet$ (i) \textit{$\Sigma$ is a spanning subgraph of $\G$ with components $\Sigma_1,\ldots,\Sigma_m$, and  
there exists an odd integer $t > 1$ such that $\Sigma_i \cong CL_t$ 
for each $i \in [1,m]$ (hence $|G|=2tm$).}
\medskip

It is clear that $\Sigma$ is a spanning graph and each component of $\Sigma$ is 
isomorphic to the graph $\cay(K,\{x,a,a^{-1}\})$, which is clearly isomorphic to 
$CL_p$. There are $6k$ components and we list them as $\Sigma_1,\ldots,\Sigma_{6k}$ 
(hence $m=6k$).
\medskip

\noindent $\bullet$ (ii) \textit{$\Sigma \cup \Lambda$ is a parity subgraph of $\G$.}  
\medskip

Let $u \in G$. It is obvious that as a vertex of $\Sigma \cup \Lambda$, 
the degree of $u$ is equal to $5$ if $u \in L \cup aL$ and $3$ otherwise, so 
$\G-\edge(\Sigma \cup \Lambda)$ is even.
\medskip

\noindent $\bullet$ (iii) \textit{For each $i \in [1,6k]$, 
$|\edge(\Lambda) \cap \edge(\Sigma_i)| \in \{1,2\}$ and 	
every edge in $\edge(\Lambda) \cap \edge(\Sigma_i)$ is a rung of $\Sigma_i$.}
\medskip 

The graph $\Lambda$ has two isomorphic components induced by the subgroup 
$L$ and its coset $aL$, respectively. Part of $\Lambda$ is depicted in Figure~\ref{fig1}.  
It is clear that the $x$-edges of $\Lambda$ are the common edges between 
$\Lambda$ and $\Sigma$, hence the second part of (iii) holds.

Fix a vertex $u \in L$. Then $u=x_\alpha y^j$ for some $\alpha \in [0,3]$ and 
$j \in [0,3k-1]$. Let $\Sigma_i$ be the component of $\Sigma$ containing $u$, where $i \in [1,6k]$. 
Then the edge $\{u,ux_1\}$ lies in $\edge(\Lambda) \cap \edge(\Sigma_i)$. 
If $u^\prime=x_{\alpha'}y^{j^\prime}$ is another vertex of $\Lambda[L]$, $u^\prime \ne u$, then $u^\prime \in \vertex(\Sigma_i)$ if and only if 
$uK=u^\prime K$. This reduces to $(x_\alpha x_\alpha^\prime)^{y^j}y^{j^\prime-j} \in K$, which shows that 
$j^\prime=j$ and $(x_\alpha x_\alpha^\prime)^{y^j}=x_1$, hence 
$u^\prime=ux_1$. We have shown that $\{u,ux_1\}$ is the only common edge between $\Lambda[L]$ and $\Sigma_i$. 

Since $auK=uK$, $au \in \vertex(\Sigma_i)$ also holds. The argument, used in the 
previous paragraph, can be copied to show that 
$\{au,aux_1\}$ is the only common edge between $\Lambda[aL]$ and $\Sigma_i$. 
Using also that $\Lambda$ has $12k$ $x$-edges and $\Sigma$ consists of $6k$ components, 
we conclude that the $12k$ $x$-edges are distributed equally among the 
components of $\Sigma$, implying that the condition in (iii) is satisfied.
\medskip

\noindent $\bullet$ (iv) \textit{$\Lambda$ admits a $\Z_3$-flow $(D,\varphi)$ such that 
\begin{enumerate}[(a)]
\item $\Null(\varphi) \subseteq \edge(\Lambda) \cap \edge(\Sigma)$. 
\item For every $i \in [1,6k]$, 
$\edge(\Lambda) \cap \edge(\Sigma_i) \not\subset \Null(\varphi)$.
\end{enumerate}}

The component $\Lambda[L]$ of $\Lambda$ is depicted in Figure~\ref{fig1}.  
We use $\Lambda[L]$ to define a new graph $\G'$.  For every $i \in [0,3k-1]$, delete the 
$x$-edges $\{y^{3i},x_1y^{3i}\}$ and $\{x_2y^{3i},x_3y^{3i}\}$ from $\Lambda[L]$; 
and then replace the induced $2$-path $(x_\alpha y^{3i-1},x_\alpha y^{3i},x_\alpha y^{3i+1})$ in the obtained graph with the edge $\{x_\alpha y^{3i-1},x_\alpha y^{3i+1}\}$, 
see Figure~\ref{fig1}. 
 
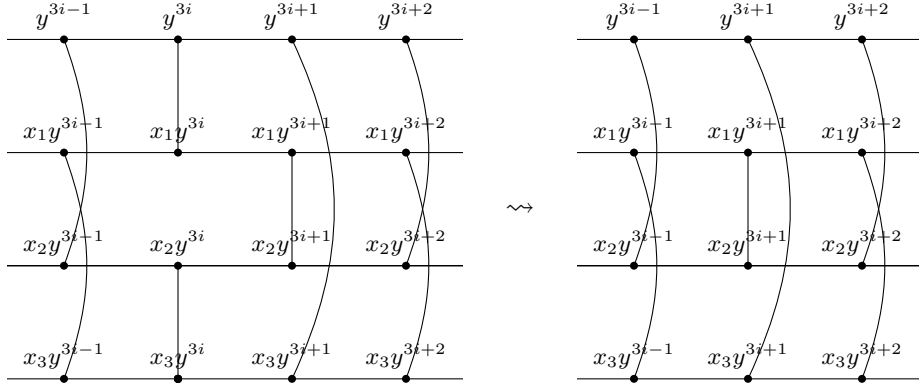
\begin{figure}[t]
\begin{center}
\begin{tikzpicture}[scale=1.5] ---
% vertices
\fill (0,0) circle (1pt);\fill (1,0) circle (1pt);
\fill (0,1) circle (1pt);
\fill (0,2) circle (1pt);
\fill (0,3) circle (1pt);
\draw (0,0) node[above]{\footnotesize $x_3y^{3i-1}$};
\draw (0,1) node[above]{\footnotesize $x_2y^{3i-1}$};
\draw (0,2) node[above]{\footnotesize $x_1y^{3i-1}$};
\draw (0,3) node[above]{\footnotesize $y^{3i-1}$};
\fill (1,0) circle (1pt);\fill (1,0) circle (1pt);
\fill (1,1) circle (1pt);
\fill (1,2) circle (1pt);
\fill (1,3) circle (1pt);
\draw (1,0) node[above]{\footnotesize $x_3y^{3i}$};
\draw (1,1) node[above]{\footnotesize $x_2y^{3i}$};
\draw (1,2) node[above]{\footnotesize $x_1y^{3i}$};
\draw (1,3) node[above]{\footnotesize $y^{3i}$};
\fill (2,0) circle (1pt);
\fill (2,1) circle (1pt);
\fill (2,2) circle (1pt);
\fill (2,3) circle (1pt);
\draw (2,0) node[above]{\footnotesize $x_3y^{3i+1}$};
\draw (2,1) node[above]{\footnotesize $x_2y^{3i+1}$};
\draw (2,2) node[above]{\footnotesize $x_1y^{3i+1}$};
\draw (2,3) node[above]{\footnotesize $y^{3i+1}$};
\fill (3,0) circle (1pt);
\fill (3,1) circle (1pt);
\fill (3,2) circle (1pt);
\fill (3,3) circle (1pt);
\draw (3,0) node[above]{\footnotesize $x_3y^{3i+2}$};
\draw (3,1) node[above]{\footnotesize $x_2y^{3i+2}$};
\draw (3,2) node[above]{\footnotesize $x_1y^{3i+2}$};
\draw (3,3) node[above]{\footnotesize $y^{3i+2}$};
\fill (5,0) circle (1pt);\fill (1,0) circle (1pt);
\fill (5,1) circle (1pt);
\fill (5,2) circle (1pt);
\fill (5,3) circle (1pt);
\draw (5,0) node[above]{\footnotesize $x_3y^{3i-1}$};
\draw (5,1) node[above]{\footnotesize $x_2y^{3i-1}$};
\draw (5,2) node[above]{\footnotesize $x_1y^{3i-1}$};
\draw (5,3) node[above]{\footnotesize $y^{3i-1}$};
\fill (6,0) circle (1pt);\fill (1,0) circle (1pt);
\fill (6,1) circle (1pt);
\fill (6,2) circle (1pt);
\fill (6,3) circle (1pt);
\draw (6,0) node[above]{\footnotesize $x_3y^{3i+1}$};
\draw (6,1) node[above]{\footnotesize $x_2y^{3i+1}$};
\draw (6,2) node[above]{\footnotesize $x_1y^{3i+1}$};
\draw (6,3) node[above]{\footnotesize $y^{3i+1}$};
\fill (7,0) circle (1pt);
\fill (7,1) circle (1pt);
\fill (7,2) circle (1pt);
\fill (7,3) circle (1pt);
\draw (7,0) node[above]{\footnotesize $x_3y^{3i+2}$};
\draw (7,1) node[above]{\footnotesize $x_2y^{3i+2}$};
\draw (7,2) node[above]{\footnotesize $x_1y^{3i+2}$};
\draw (7,3) node[above]{\footnotesize $y^{3i+2}$};
% edges 
\draw (-0.5,0) -- (3.5,0) (-0.5,1) -- (3.5,1) (-0.5,1) -- (3.5,1) (-0.5,2) -- (3.5,2) (-0.5,3) -- (3.5,3); 
\draw (4.5,0) -- (7.5,0) (4.5,1) -- (7.5,1) (4.5,1) -- (7.5,1) (4.5,2) -- (7.5,2) (4.5,3) -- (7.5,3); 
% x-
\draw[bend left=-20] (0,0) to (0,2); 
\draw[bend left=-20] (0,1) to (0,3); 
\draw (1,0) -- (1,1) (1,2) -- (1,3); 
\draw (2,1) -- (2,2);
\draw[bend left=-25] (2,0) to (2,3);
\draw[bend left=-20] (3,0) to (3,2); 
\draw[bend left=-20] (3,1) to (3,3); 
\draw[bend left=-20] (5,0) to (5,2); 
\draw[bend left=-20] (5,1) to (5,3); 
\draw (6,1) -- (6,2);
\draw[bend left=-25] (6,0) to (6,3);
\draw[bend left=-20] (7,0) to (7,2); 
\draw[bend left=-20] (7,1) to (7,3); 
%\draw[line width=2] (1,0.8) -- (1,1.6) (6,0.8) -- (6,1.6) (1,4.2) -- (1,5) (6,4.2) -- (6,5);  
\draw (4,1.5) node{$\rightsquigarrow$};
\end{tikzpicture}
\caption{The graph $\Lambda[L]$ (left) and the graph $\G'$ (right) obtained from 
$\Lambda[L]$.}
\label{fig1}
\end{center}
\end{figure}

It is clear that $\G'$ is a $3$-valent graph. Define the subsets $U, W \subseteq \vertex(\G')$ as 
\begin{align*}
U &= \{y^{3i+1}, x_1y^{3i+1}, x_2y^{3i+2}, x_3y^{3i+2} \mid i \in [0,k-1]\}, \\ 
W &=\{y^{3i+2}, x_1y^{3i+2}, x_2y^{3i+1}, x_3y^{3i+1} \mid i \in [0,k-1]\}.
\end{align*}
A direct check shows that $\G'$ is also bipartite with bipartition parts $U$ and $W$, and 
therefore, $\G'$ admits a nowhere-zero $\Z_3$-flow (see~\cite[Proposition~6.4.2]{Dbook}).
Consequently,  there exists a $\Z_3$-flow 
$(D_1,\varphi_1)$ in $\Lambda[L]$ for which 
\[
\Null(\varphi_1)=\big\{\{y^{3i},x_1y^{3i}\}, \{x_2y^{3i},x_3y^{3i}\} \mid i \in [0,k-1]\big\}.
\]
Similarly, there exits a $\Z_3$-flow $(D_2,\varphi_2)$ in $\Lambda[aL]$ for which 
\[
\Null(\varphi_2)=\big\{\{ay^{3i+2},ax_2y^{3i+2}\}, \{ax_1y^{3i+2},ax_3y^{3i+2}\}
\mid i \in [0,k-1]\big\}.
\]
Then define the $\Z_3$-flow $(D,\varphi)$ in $\Lambda$ as 
for every edge $e \in \edge(\Lambda)$, if $e$ lies in $L$, then let 
$D(e)=D_1(e)$, $\varphi(e)=\varphi_1(e)$, and let $D(e)=D_2(e)$, 
$\varphi(e)=\varphi_2(e)$ otherwise. 

We claim that $(D,\varphi)$ is a required $\Z_3$-flow. It is obvious that the 
condition in (a) is satisfied. We have shown in the 
proof of (iii) that for every $i \in [1,6k]$, the two common edges between 
$\Lambda$ and $\Sigma_i$ are $\{u,ux\}$ and 
$\{au,aux\}$ for some $u \in L$. The construction of $\varphi$ shows that 
$\varphi$ cannot vanish at both edges, by which the condition in (b) is also met. 
\end{proof}

In the proof of the existence of a nowhere-zero $\Z_3$-flow in the graph $\G_2$ we 
use Lemma~\ref{new-condition}. 
In the following lemma we define a partition of 
$G$ into subsets $U$ and $W$ and show that the condition in (i) of 
Lemma~\ref{new-condition} 
holds for $\G[U]$ and $\G[W]$; and then in the subsequent lemma we also prove the fulfillment of 
the condition in (ii) of Lemma~\ref{new-condition}.

\begin{lem}\label{G2-1}
Let $\varepsilon \in \{0,1\}$ and  
\[
\H_\varepsilon=\{a^iy^{3j} \mid i \in [0,p-1], j 
\in [0,k-1], j  \equiv \varepsilon\!\!\!\!\!\pmod 2\}. 
\]
Define the subset $U \subseteq G$ as $U=X_0 \H_0 \cup X_1 \H_1$ and let 
$W=G \setminus U$. Then $\Gamma_2[U]$ and $\G_2[W]$ are pseudoforests. 
\end{lem}
\begin{proof}
By Lemma~\ref{iso}, $\G_2[U] \cong \G_2[W]$, hence we may focus on $\G_2[U]$. 
We show that each component of $\G_2[U]$ is a cycle or an edge or an isolated 
vertex.  

The connection set $X$ of $\G_2$ is equal to $\{x,ay,(ay)^{-1},y,y^{-1}\}$. 
By definition, $U$ decomposes as 
\begin{equation}\label{eq:U}
\begin{matrix}
U & = & \H_0  & \cup &  x_1\H_1 & \cup &  x_2\H_1  & \cup & x_3\H_0 & \cup \\
   &     & y\H_1 & \cup & x_1y\H_0 & \cup & x_2y\H_0  & \cup & x_3y\H_1 & \cup \\
   &     & y^2\H_0 & \cup & x_1y^2\H_1 & \cup & x_2y^2\H_1 & \cup & x_3y^2\H_0. 
\end{matrix}
\end{equation}

Fix a vertex $u \in U$. Then $u \in x_\alpha y^\ell \H_\varepsilon$ for 
some $\alpha \in [0,3]$, $\ell \in [0,2]$ and $\varepsilon \in [0,1]$. 
If $\ell \ne 2$, then $uy \in x_\alpha y^{\ell+1}\H_{\varepsilon}$, hence 
$uy \notin U$ (see~\eqref{eq:U}). Let $\ell=2$ and $u=x_\alpha y^{2} (a^{i} y^{3j})$, where 
$i \in [0,p-1]$ and $j \in [0,k-1]$. 
If $j < k-1$, then $uy \in x_\alpha \H_{1-\varepsilon}$, hence 
$uy$ is outside $U$ again. If $j=k-1$, then $u=a^{r^{-2}i} y^{-1}$ or 
$x_3a^{r^{-2}i} y^{-1}$, where $r^{-1}$ is the unique integer in $[1,p-1]$ that satisfies 
$rr^{-1} \equiv 1\!\!\pmod p$, and hence $uy \in U$. We conclude that  
the $y$-edges of $\G[U]$ can be listed as 
$\{a^iy^{-1},a^i\}$ and $\{x_3a^iy^{-1},x_3a^i\}$, 
$i \in [0,p-1]$.  Repeating the argument gives that 
$uay \in U$ if and only $u=a^iy^{-1}$ or $x_3a^iy^{-1}$ for some $i \in [0,p-1]$, and 
using also \eqref{eq:y-on-ai}, we find that the $ay$-edges of $\G_2[U]$ are exactly 
$\{x_\alpha a^iy^{-1},x_\alpha a^{i+r}\}$, $\alpha=0$ or $3$ and $i \in [0,p-1]$; 
and therefore, the $y$- and $ay$-edges induce two disjoint $2p$-cycles in 
$\G_2[U]$. 

We turn to the $x$-edges of $\G_2[U]$. These clearly form a matching. 
It follows from \eqref{eq:y-on-S} and \eqref{eq:U} that $ux \in U$ if and only if  $\ell=1$. 
Thus no $x$-edge is incident with a $y$- or an $ay$-edge, hence  
$\G_2[U]$ is a disjoint union of cycles, isolated vertices and edges.  
\end{proof}

\begin{lem}\label{G2-2}
Assuming the notations in Lemma~\ref{G2-1}, there exist a transversal 
$U'$ of $\G_2[U]$ and a transversal $W'$ of $\G_2[W]$ such that 
$\G_2[U',W']$ has a perfect matching. 
\end{lem} 
\begin{proof}
We begin with some properties of $G$. Let $P=\sg{a}$. Then $P < H < G$ and the elements 
$x_\alpha y^i$, $\alpha \in [0,3]$ and $i\in [0,2]$, form a transversal of $H$ in $G$. 
Furthermore, $H=\H_0 \cup \H_1$ and we have the following identities: 
\begin{equation}\label{eq:H's}
\H_1y^3 =\H_0 \setminus P~\text{and}~\H_0y^3=\H_1 \cup P.
 \end{equation}

Let us turn to the graph $\G_2[U]$. We have seen in the proof of Lemma~\ref{G2-1} 
that $\G_2[U]$ consists of two components, which are cycles of length $2p$ and cover 
the vertices in the set $P \cup Py^{-1} \cup x_3P \cup x_3Py^{-1}$; components, which are edges, 
and these join the vertices in the set $y\H_1 \cup  x_1y\H_0$ with the vertices in the set $x_3y\H_1 \cup  x_2y\H_0$; and isolated vertices. This shows that the subset 
\begin{align*}
U':= & ~ \H_0~\cup~x_1\H_1~\cup~x_2\H_1~\cup~x_3\H_0~\cup~y\H_1\cup~x_1y\H_0~
\cup 
\\
& ~y^2\H_0~\cup~x_1y^2\H_1~\cup~x_2y^2\H_1~\cup~x_3y^2\H_0~  
\setminus~(P~\cup~x_3P~\cup~Py^{-1}~\cup~x_3Py^{-1})
\end{align*}
is a transversal of $\G_2[U]$. We settle the lemma by finding   
a mapping $\nu : U' \to G$ so that %$\nu$ is injective,  
$\nu(U')$ is a transversal of $\G_2[W]$ and 
for every $u \in U'$, $u$ and $\nu(u)$ are adjacent in $\G_2$.  For this purpose, 
partition $U'$ into subsets $U'_x$, $U'_y$ and $U'_{y^{-1}}$ defined as  
\begin{align*}
U'_x=&~\H_0\cup~x_1\H_1~\setminus~(P~\cup~Py^{-3}), \\
U'_y=& ~Py^{-3}~\cup~y \H_1~\cup~x_1y\H_0~\cup~x_2y^2\H_1~\cup~
x_3y^2\H_0~\setminus~(x_1Py^{-2}~\cup~x_3Py^{-1}),\\
U'_{y^{-1}}=& ~ x_1Py^{-2}~\cup~y^2\H_0~\cup~x_1y^2\H_1~\cup~x_2\H_1
~\cup~x_3\H_0~\setminus~(x_3P~\cup~Py^{-1}).
\end{align*}

Then define the function $\nu : U' \to G$ as for every $u \in U'$, 
$\nu(u)=uz$ if $u \in U'_{z}$ ($z \in \{x,y,y^{-1}\}$). We claim that $\nu$ has the required properties. 
First, by construction, $u \sim \nu(u)$ for every $u \in U'$. 
Using~\eqref{eq:H's}, we compute that  
\begin{align*}
\nu(U') = & ~ \Big( x_1\H_0~\cup~\H_1~\setminus~(x_1P~\cup~x_1Py^{-3}) \Big)~\cup \\
& ~  \Big( Py^{-2}~\cup~y^2 \H_1~\cup~x_1y^2\H_0~\cup~x_2\H_0~\cup~
x_3\H_1~\setminus~(x_2P~\cup~x_1Py^{-1}) \Big)~\cup \\
& ~ \Big( x_1Py^{-3}~\cup~y\H_0~\cup~x_1y\H_1~\cup~x_2y^2\H_0~\cup~x_3y^2\H_1~
\setminus (x_2Py^{-1} \cup Py^{-2}) \Big) \\
= &~\H_1~\cup~x_1\H_0~\cup~x_2\H_0~\cup~x_3\H_1~\cup~y\H_0~
\cup~x_1y\H_1~\cup \\ 
&~y^2\H_1~\cup~x_1y^2\H_0~\cup~x_2y^2 \H_0~\cup~ x_3y^2\H_1~\setminus~
\big(x_1P~\cup~x_2P~\cup~x_1Py^{-1}~\cup~x_2Py^{-1}\Big). \\ 
= &~ x_1U'.
\end{align*}
Since $x_1U'$ is a transversal of $\G_2[W]$, it follows that $\nu$ is indeed a desired 
mapping.
\end{proof}
%---------------------------------------------------------------------------------------------------%
\subsection{Graphs $\Gamma_3$ and $\Gamma_4$}\label{sec:4.2}

In this case, $G \cong A_4 \times \Z_p$, 
$y \in G$ and $|y|=3p$. 
The action of $y$ on the Sylow $2$-subgroup $S$ of $G$ is 
defined in \eqref{eq:y-on-S}. Note that $H=\sg{a,y^3}=\sg{y^3}$ is the Sylow 
$p$-subgroup of $G$.  

Every element of $G$ is uniquely written as the product 
$x_\alpha y^i a^j$, where $\alpha \in [0,3]$, $i \in [0,2]$ and $j \in [0,p-1]$. 
If $y^3=a^s$ for some $s\in [1,p-1]$, then the edges of $\G_3 \cup \G_4$ can be listed as
\begin{equation}\label{eq:edges}
\begin{array}{clll}
\text{($x$-edges)} & 
x_\alpha a^j \sim x_1 x_\alpha a^j & x_\alpha ya^j \sim x_3 x_\alpha ya^j & 
x_\alpha y^2a^j  \sim x_2 x_\alpha y^2a^j, \\ 
\text{($y$-edges)} &
x_\alpha a^j   \sim x_\alpha ya^j & x_\alpha ya^j \sim x_\alpha y^2a^j &  
x_\alpha y^2a^j \sim x_\alpha a^{j+s},  \\ 
\text{($ay$-edges)} &
x_\alpha a^j  \sim x_\alpha ya^{j+1} & x_\alpha ya^j \sim x_\alpha y^2a^{j+1} & 
x_\alpha y^2a^j \sim x_\alpha a^{j+s+1},  \\
\text{($a$-edges)} &  x_\alpha a^j \sim x_\alpha a^{j+1} & 
x_\alpha ya^j \sim x_\alpha ya^{j+1} & 
x_\alpha y^2a^j \sim x_\alpha y^2a^{j+1}.
\end{array}
\end{equation}

We prove the existence of a nowhere-zero $\Z_3$-flow in both $\G_3$ and $\G_4$ 
using Lemma~\ref{new-condition}.  First we deal with $\G_4$ in the special case 
where $y^3=a^{-1}$. 

\begin{lem}\label{G4--1}
$\G_4$ admits a nowhere-zero $\Z_3$-flow, provided that $y^3=a^{-1}$.
\end{lem}
\begin{proof}
In this case, the connection set $X$ is equal to 
$\{x,y,y^{-1},y^2,y^{-2}\}$. 
Recall that $S=\{x_0,x_1,x_2,x_3\}$ is the Sylow $2$-subgroup of $G$, where 
$x_0=1$ and $x_1=x$. We distinguish two cases according to the residue of $p$ modulo $4$. 
\medskip

\noindent{\bf Case~1.} $p \equiv 1\!\!\pmod 4$.\ Define the subsets $V_1, V_2 \subseteq G$ as 
\begin{align*}
V_1=& \{1,x_2\}\big\{ y^i \mid i \in [0,3p-6],~
i \equiv 0,1\!\!\!\!\!\pmod 4 \big\}~\cup \\ 
& \{x_1,x_3\}\big\{ y^i \mid i \in [0,3p-6],~
i \equiv 2,3\!\!\!\!\!\pmod 4 \big\}, \\ 
V_2=& \{y^{3p-3},x_2p^{3p-3}\}~\cup~
\{ x_1,x_3\}\big\{y^i \mid i \in \{3p-5,3p-4,3p-2,3p-1\}\big\}. 
\end{align*} 
Then let $U=V_1 \cup V_2$ and $W=G \setminus U$.
In view of Lemma~\ref{new-condition}, it is sufficient to show that 
$\G_4[U]$ and $\G_4[W]$ are pseudoforests and there exist 
a transversal $U'$ of $\G_4[U]$ and a transversal $W'$ of $\G_4[W]$ such that $\G_4[U',W']$ has a perfect matching. 

Let $i \in [0,(3p-7)/2]$. It is straightforward to check 
that every edge of $\G_4[U]$ lies entirely in the set 
$Sy^{2i} \cup Sy^{2i+1}$ or its complement, and the subgraph 
$\G_4[(Sy^{2i} \cup Sy^{2i+1}) \cap U]$ is either a path of length $3$ or the union of two independent edges.  
Note that $\G_4[W]=x_1\G_4[U]$ and $x_1(Sy^{2i} \cup Sy^{2i+1})=Sy^{2i} \cup Sy^{2i+1}$. 
Thus, if $\G'$ is a component of $\G_4[(Sy^{2i} \cup Sy^{2i+1}) \cap U]$, then $x_1\G'$ is a component of $\G_4[(Sy^{2i} \cup Sy^{2i+1}) \cap W]$. Furthermore, one can find 
vertices $u, v \in \vertex(\G')$ (not necessarily distinct) such that $u \sim x_1v$ in $\G_4$.
This means that there exist a transversal $U_i'$ of $\G_4[(Sy^{2i} \cup Sy^{2i+1}) \cap U]$ and a transversal $W_i'$ of $\G_4[(Sy^{2i} \cup Sy^{2i+1}) \cap W]$
such that $\G_4[U_i',W_i']$ has a perfect matching. 

It remains to consider the subgraphs $\G_4[V_2 \cap U]$ and $\G_4[V_2 \cap W]$. 
We compute that each of these graphs has one unicyclic component on $8$ vertices and 
and two isolated vertices. The isolated vertices of $\G_4[V_2 \cap U]$ are 
$y^{3p-3}$ and $x_2y^{3p-3}$ and the isolated vertices of $\G_4[V_2 \cap W]$ are  $x_1y^{3p-3}$ and $x_3y^{3p-3}$. All these yield that $\G_4[U]$ and $\G_4[W]$ are pseudoforests and the subsets $U' \subseteq U$ and $W' \subseteq W$, defined as  
$U'=\bigcup_{i=0}^{(3p-7)/2} U_i'~\cup~
\{y^{3p-3}, x_2y^{3p-3}\}$ and 
$W'=\bigcup_{i=0}^{(3p-7)/2}W_i'~\cup~\{x_1y^{3p-3}, x_3y^{3p-3}\}$ are 
the required transversals. 
\medskip

\noindent{\bf Case~2.} $p \equiv 3\!\!\pmod 4$.\ Define the subsets $V_1, V_2 \subseteq G$ as 
\begin{align*}
V_1=& \{1,x_2\}\big\{ y^i \mid i \in [0,3p-16],~
i \equiv 0,1\!\!\!\!\!\pmod 4 \big\}~\cup \\ 
& \{x_1,x_3\}\big\{ y^i \mid i \in [0,3p-16],~
i \equiv 2,3\!\!\!\!\!\pmod 4 \big\}, \\ 
V_2=& \{1,x_2\}\{y^i \mid i \in \{3p-13,3p-10,3p-9,3p-7,3p-6,3p-3\} \}~\cup \\
& \{ x_1,x_3\}\big\{y^i \mid i \in \{3p-15,3p-14,3p-12,3p-11,3p-8,3p-5,3p-4,\\ 
&~3p-2,3p-1\}\big\}. 
\end{align*} 
Then let $U=V_1 \cup V_2$ and $W=G \setminus U$. Again, our goal is to show that 
$\G_4[U]$ and $\G_4[W]$ are pseudoforests and there exist 
a transversal $U'$ of $\G_4[U]$ and a transversal $W'$ of $\G_4[W]$ such that $\G_4[U',W']$ has a perfect matching. It is clear that the argument used in Case~1 can be repeated to deduce that for 
every $i \in [0,(3p-17)/2]$, the subgraphs $\G_4[(Sy^{2i} \cup Sy^{2i+1}) \cap U]$ and 
$\G_4[(Sy^{2i} \cup Sy^{2i+1}) \cap W]$ is a union of components of $\G_4[U]$ and 
$\G_4[W]$, respectively, and also that there exist a transversal $U_i'$ of $\G_4[(Sy^{2i} \cup Sy^{2i+1}) \cap U]$ and a transversal $W_i'$ of $\G_4[(Sy^{2i} \cup Sy^{2i+1}) \cap W]$
such that $\G_4[U_i',W_i']$ has a perfect matching. 

Then, we compute that $\G_4[V_2 \cap U]$ (and $\G_4[V_2 \cap W]$ as well) consists of two 
unicyclic components on $8$ vertices both, two path components 
of length $3$ both, one edge component, and four isolated vertices. 
The two path components of 
$\G_4[V_2 \cap U]$ are 
\[
(x_\alpha,y^{p-15},x_\alpha y^{p-14},x_\alpha y^{p-12},x_\alpha y^{p-11}),~\alpha \in \{x_1,x_3\},
\]
the edge component is $\{y^{p-13},x_2p^{p-13}\}$ and  
the four isolated vertices are $x_1y^{p-8}, x_3y^{p-8}, y^{p-3}$ and  $x_2y^{p-3}$. 
Eventually, we conclude that $\G_4[U]$ and $\G_4[W]$ are pseudoforests and the 
subsets $U' \subset U$ and $W' \subset W$, defined as  
\begin{align*}
U'=&\bigcup_{i=0}^{(3p-17)/2}U_i'~\cup~
\{x_3y^{p-15}, x_1y^{p-14}, y^{p-13}, x_1y^{p-8}, x_3y^{p-8}, y^{p-3},x_2y^{p-3}\}, \\
W'=&\bigcup_{i=0}^{(3p-17)/2}W_i'~\cup~
\{x_2y^{p-15}, y^{p-14}, x_1y^{p-13}, y^{p-8}, x_2y^{p-8}, x_1y^{p-3},x_3y^{p-3}\}
\end{align*}
are the required transversals. 
\end{proof}

Let $Y=\sg{y}$ and $\Sigma=\cay(Y,X \setminus \{x\})$. Recall that $Y\cong \Z_{3p}$. 
Clearly, $\Sigma$ is a subgraph of both $\G_3$ and $\G_4$. In the next two lemmas, we develop conditions on 
$\Sigma$ that ensure that $\G_3$ and $\G_4$ satisfy the conditions of Lemma~\ref{new-condition}. Although these conditions are technical to state, they will be easy to 
verify for the Cayley graph $\Sigma$. 
%\medskip

In what follows, for a subgraph $\Sigma'$ of $\Sigma$ and vertex $v \in \vertex(\Sigma')$, 
the component of $\Sigma'$ containing $v$ is denoted by $\Sigma'(v)$.  

\begin{lem}\label{sigma's-1}
Assuming the notations in case (ii) of Lemma~\ref{reduction}, let  
$\G=\G_3$ or $\G_4$, $Y=\sg{y}$ and $\Sigma=\cay(Y,X \setminus \{x\})$. 
Suppose that $\H \subset H$ is a non-empty proper subset such that for the set  
$\Y=\H \cup y(H \setminus \H) \cup y^2\H$, the graphs $\Sigma_0=\Sigma[\Y]$ and $\Sigma_1=\Sigma[Y \setminus \Y]$  satisfy the following conditions.
\begin{enumerate}[(i)]
% (i)
\item $\Sigma_0$ and $\Sigma_1$ are pseudoforests.
% (ii) 
\item 
For every $h \in \H$, $\Sigma_1(yh)$ is a tree and  
$|\vertex(\Sigma_1(yh)) \cap yH| \le 2$. 
\item 
For every $h \in H \setminus \H$, 
$\Sigma_0(yh)$ is a tree and  
$|\vertex(\Sigma_0(yh)) \cap yH| \le 2$. 
% (iii) 
\end{enumerate}
Then for the set $U=X_0\H \cup X_1(H \setminus \H)$, 
$\G[U]$ and $\G[G \setminus U]$ are pseudoforests.
\end{lem}
\begin{proof}
By Lemma~\ref{iso}, it is sufficient to show that $\G[U]$ is a pseudoforest. 
Let $\H_0=\H$ and $\H_1=H \setminus \H$. For $i \in [0,1]$ and $\alpha \in [1,3]$, let $x_\alpha\Sigma_i$ denote the subgraph of $\G$ whose vertex set is $x_\alpha Y$ 
and edge set is $x_\alpha\edge(\Sigma_i)$. 

Recall that $S$ forms a transversal of $Y$ in $G$.  
One can readily deduce, using \eqref{eq:edges}, that every $x$-edge of $\G[U]$ joins either a 
vertex in $Y$ with a vertex in $x_3Y$ or a vertex in $x_1Y$ with a vertex in $x_2Y$, and 
the remaining edges of $\G[U]$ lies entirely in the cosets $x_\alpha Y$, where $\alpha \in [0,3]$.
In particular, $\G[U\cap Y]=\Sigma[U \cap Y]$ and $\G[U \cap x_1Y]=x_1\Sigma[x_1U \cap Y]$. 
On the other hand,
\begin{align*}
U \cap Y =& (X_0\H_0 \cap Y) \cup (X_1\H_1 \cap Y)=\H_0 \cup y^2\H_0 \cup y\H_1=\Y \\ 
x_1U \cap Y=& (X_0\H_1 \cap Y) \cup (X_1\H_0 \cap Y)=\H_1 \cup y^2\H_1 \cup y\H_0=Y \setminus \Y. 
\end{align*}
All these together with the identity $x_3U=U$ yield that, if 
$E$ denotes the set of $x$-edges of $\G[U]$, then  
\begin{equation}\label{eq:-E}
\G[U]-E=\bigcup_{\alpha=0}^{3}\G[U \cap x_\alpha Y]=
\Sigma_0~\cup~x_3\Sigma_0~\cup~x_1\Sigma_1~\cup~x_2\Sigma_1.
\end{equation}

Let $\Sigma'$ be any component of $\G[U]$. We have to show that $\Sigma'$ contains at most 
one cycle.  If $\Sigma'$ has no $x$-edge, then $\Sigma'$ is a component of one of 
the four graphs in the right side of \eqref{eq:-E}. Then the condition in (i) guarantees that $\Sigma'$ contains at most one cycle.
Now, assume that $\Sigma'$ contains an $x$-edge, say $e$. 
Then $e$ is either an edge between $Y=\vertex(\Sigma_0)$ and $x_3Y=\vertex(x_3\Sigma_0)$ or an 
edge between $x_1Y=\vertex(x_1\Sigma_1)$ and $x_2Y=\vertex(x_2\Sigma_1)$. If the former case occurs, 
then by the condition in (iii), $\Sigma'[Y]$ is a tree and 
$|\vertex(\Sigma') \cap yH|=1$ or $2$. 
Using also that $\Sigma'[x_3Y]=x_3\Sigma'[Y]$, we find that  
$\Sigma'$ is a tree if $|\vertex(\Sigma') \cap yH|=1$, and  
$\Sigma'-\{e\}$ is a tree if $|\vertex(\Sigma') \cap yH|=2$. 
In either case, $\Sigma'$ contains at most one cycle, as required. 
Finally, if $e$ is ad edge between $x_1Y$ and $x_2Y$, then  
the same argument can be repeated using the condition in (ii) to deduce that $\Sigma'$ contains at most one cycle.
\end{proof}

Assume the notations and conditions in the previous lemma. 
Let $\varepsilon \in [0,1]$ and $\Sigma'$ be a component of $x_\varepsilon\Sigma_\varepsilon$.
The proof of the lemma reveals that the subgraph $\G'$ of $\G$ ($\G=\G_3$ or $\G_4$), 
defined as 
\begin{equation}\label{eq:comp}
\G'=\begin{cases}
\Sigma'~\text{or}~x_3\Sigma' & \text{if}~
\vertex(\Sigma') \cap x_\varepsilon yH=\emptyset,  \\
(\Sigma' \cup x_3\Sigma')+\{e\} & \text{if}~
\vertex(\Sigma') \cap x_\varepsilon yH=\{x_\varepsilon yh\}, \\ 
(\Sigma' \cup x_3\Sigma')+\{e,e'\} & \text{if}~
\vertex(\Sigma') \cap x_\varepsilon yH=\{x_\varepsilon yh,x_\varepsilon yh'\}
\end{cases}
\end{equation}
is a component of $\G[U]$, where $e$ and $e'$ are the edges $\{x_\varepsilon yh,x_3 x_\varepsilon yh\}$ and $\{x_\varepsilon yh',x_3x_\varepsilon yh'\}$, respectively,  
$(\Sigma' \cup x_3\Sigma')+\{e\}$ is the graph obtained from $\Sigma' \cup x_3\Sigma'$ by adding the edge $e$ to it, and $(\Sigma' \cup x_3\Sigma')+\{e,e'\}$ is defined similarly. 
Furthermore, all components of $\G[U]$ are obtained in this way. 

In what follows, we call a subset $T \subset \vertex(\Sigma_\varepsilon)$ a \emph{reduced transversal} if $T$ contains exactly one vertex from each component $\Sigma'$ of $\Sigma_\varepsilon$ if $\Sigma'$ is a tree and 
$|\vertex(\Sigma') \cap yH| \le 1$, and no vertex from the other components. 

\begin{lem}\label{sigma's-2}
Assuming the notations and the conditions in Lemma~\ref{sigma's-1}, 
suppose that there are reduced transversals $T_0$ of $\Sigma_0$ and 
$T_1$ of $\Sigma_1$ such that for $\varepsilon \in [0,1]$, the following conditions hold.
\begin{enumerate}[(i)]
%(i)
\item If $u \in T_\varepsilon \cap y^2H$, then  
$\vertex(\Sigma_\varepsilon(u)) \cap yH=\emptyset$.
% (ii)
\item If $\G=\G_3$, then $y(T_\varepsilon \cap H)=T_{1-\varepsilon} \cap yH$.
% (iii)
\item If $\G=\G_4$, then 
there exist disjoint subsets $T_{\varepsilon,1}, T'_{\varepsilon,1} \subseteq T_\varepsilon \cap H$ and disjoint subsets $T_{\varepsilon,-1}, T'_{\varepsilon,-1} \subseteq T_\varepsilon \cap y^2H$ for which 
\[
yT_{\varepsilon,1} \cup ayT'_{\varepsilon,1} \cup 
y^{-1}T_{\varepsilon,-1} \cup (ay)^{-1}T'_{\varepsilon,-1} = 
T_{1-\varepsilon} \cap yH,
\]
and the subsets in left sides above are pairwise disjoint. 
\end{enumerate}
Then there exist a transversal $U'$ of $\G[U]$ and a transversal 
$W'$ of $\G[G \setminus U]$ such that 
$\G[U',W']$ has a perfect matching. 
\end{lem}
\begin{proof} 
In seeking for a suitable transversal of $\G[U]$ and a suitable transversal of $\G[G \setminus U]$, 
we introduce a couple of subsets of $T_0$ and $T_1$. First, for $i \in [0,2]$, 
let 
\[
A_i =T_0 \cap y^iH~\text{and}~B_i =T_1 \cap y^iH.
\]
It follows from \eqref{eq:comp} and the definition of a reduced transversal that the set 
$T_0$ covers the tree components of $\G[U]$ containing a vertex from $Y$ and  
$x_1T_1$ covers the tree components a vertex from $x_1Y$. Let
\[
A_0^*=\{u \in A_0 \mid \vertex(\Sigma_0(u)) \cap yH=\emptyset \}~\text{and}~
B_0^*=\{u \in B_0 \mid \vertex(\Sigma_1(u)) \cap yH=\emptyset \}.
\]
It follows from \eqref{eq:comp} and the condition in (i) that 
the set $x_3(A_0^* \cup A_2)$ covers the components of $\G[U]$ that lie entirely 
in $x_3Y$ and the set $x_2(B_0^* \cup B_2)$ covers the components of $\G[U]$ that lie entirely 
in $x_2Y$. Therefore, the set 
\[
U'=T_0~\cup~x_1T_1~\cup~x_2\big[B_0^*~\cup~B_2\big]~\cup~x_3\big[A_0^* \cup A_2]
\]
is a transversal of $\G[U]$. 

We deal first with the case where $\G=\G_3$. 
The conditions in (ii) read as 
$B_1=yA_0$ and $A_1=yB_0$. Thus, $T_0=A_0 \cup yB_1 \cup A_2$ and 
$T_1=B_0 \cup yA_1 \cup B_2$. 
Then, define the mapping $\nu : U' \to G$ as for $u \in U'$, let 
\[
\nu(u)=\begin{cases}
ux & \text{if}~u \in A_2~\cup~x_1B_2~\cup~
x_2\big[B_0^*~\cup~B_2\big]~\cup~x_3\big[A_0^*~\cup~A_2\big], \\ 
uy & \text{if}~u \in A_0~\cup~x_1B_0, \\
uy^{-1} & \text{if}~u \in yB_0 \cup x_1yA_0. 
\end{cases}
\]
To finish the proof it is sufficient to show that $\nu(U')$ is a transversal of
$\G[G \setminus U]$. Using the action of $y$ on $S$, see~\eqref{eq:y-on-S}, we compute that 
\begin{align*}
\nu(U')=&~
yA_0 \cup B_0~\cup~x_2A_2~\cup~x_1\big[yB_0~\cup~A_0~\cup~x_2B_2 \big]~\cup~x_2\big[x_1B_0^*~\cup~x_2B_2]~\cup \\
&~x_3\big[x_1A_0^*~\cup~x_2A_2\big] \\
=&~T_1~\cup~x_1T_0~\cup~x_2\big[A_0^*~\cup~A_2]~\cup~x_3\big[B_0^*~\cup~B_2\big]=xU'.
\end{align*}
This shows that $\nu(U')$ is a transversal of $\G[G \setminus U]$, as required.
\medskip

We turn to the case where $\G=\G_4$. We introduce some additional subsets of $T_0$ and $T_1$: 
\begin{enumerate}[]
\item $A_0'=T_{0,1}$,~$A_0''=T'_{0,1}$,~$A_0'''=A_0 \setminus (A_0'~\cup~A_0'')$, 
\item $A_2'=T_{0,-1}$,~$A_2''=T'_{0,-1}$,~$A_2'''=A_2 \setminus (A_2'~\cup~A_2'')$, 
\item $B_0'=T_{1,1}$,~$B_0''=T'_{1,1}$,~$B_0'''=B_0 \setminus (B_0'~\cup~B_0'')$, 
\item $B_2'=T_{1,-1}$, $B_2''=T'_{1,-1}$,~$B_2'''=B_2 \setminus (B_2'~\cup~B_2'')$.
\end{enumerate}
Then the conditions in case (iii) of the lemma read as 
\begin{align*}
A_1=&~yB_0'~\cup~ayB_0''~\cup~y^{-1}B_2'~\cup~(ay)^{-1}B_2'',  \\
B_1=&~yA_0~'\cup~ayA_0''~\cup~y^{-1}A_2'~\cup (ay)^{-1}A_2''. 
\end{align*} 

Observe that the set  
\begin{equation}\label{eq:U'}
\begin{aligned}
U':=&~ T_0~\cup~  
x_1\big[B_0~\cup~yA_0'~\cup~ayA_0''~\cup~B_2\big]~\cup \\
&~x_2\big[y^{-1}A_2'~\cup~(ay)^{-1}A_2''~\cup~B_0^*~\cup~B_2\big]~\cup~
x_3\big[A_0^*~\cup~A_2\big]. 
\end{aligned}
\end{equation}
is a transversal of $\G[U]$. It remains to find a suitable transversal of $\G[G \setminus U]$.
For this purpose, we rewrite $U'$ as 
\begin{align*}
U'=&~A_0'~\cup~A_0''~\cup~A_0'''~\cup~yB_0'~\cup~ayB_0''~\cup  
y^{-1}B_2'~\cup~(ay)^{-1}B_2''~ \cup~A_2'~\cup A_2''~\cup~A_2'''~\cup \\ 
&~x_1\big[B_0'~\cup~B_0''~\cup~B_0'''~\cup~yA_0'~\cup~ayA_0''~\cup~B_2\big]~\cup \\
&~x_2\big[y^{-1}A_2'~\cup~(ay)^{-1}A_2''~\cup~B_0^*~\cup~
B_2'~\cup~B_2''~\cup~B_2'''\big]~\cup \\
&~x_3\big[A_0^*~\cup~A_2\big].
\end{align*}
Then, define $\nu : U' \to G$ as for $u \in U'$, let 
\[
\nu(u)=\begin{cases}
ux & \text{if}~u \in \begin{cases}
A_0''' \cup A_2'''~\cup~x_1\big[B_0'''~\cup~B_2\big]~\cup \\
x_2\big[B_0^*~\cup~B_2'''\big]~\cup~x_3\big[A^*~\cup~A_2\big], 
\end{cases} \\ 
uy & \text{if}~u \in A_0'~\cup~y^{-1}B_2'~\cup~x_1B_0'~\cup~x_2y^{-1}A_2, \\
uay & \text{if}~u \in A_0''~\cup~(ay)^{-1}B_2''~\cup~x_1B_0''~\cup~x_2(ay)^{-1}A_2'', \\
uy^{-1} & \text{if}~u \in yB_0'~\cup~A_2'~\cup~x_1yA_0'~\cup~x_2B_2', \\
u(ay)^{-1} & \text{if}~u \in ayB_0''~\cup~A_2''~\cup~x_1ayA_0''~\cup~x_2B_2''.
\end{cases}
\]
To finish the proof it is sufficient to show that $\nu(U')$ is a transversal of
$\G[G \setminus U]$. 
We compute that 
\begin{align*}
\nu(U')=&~
yA_0'~\cup~ayA_0''~\cup~x_1A_0'''~\cup~B_0'~\cup~B_0''~\cup~  
B_2'~\cup~B_2''~\cup~y^{-1}A_2'~\cup~(ay)^{-1}A_2''~\cup~x_2A_2'''~\cup  \\
&~x_1\big[A_0'~\cup~A_0''~\cup~yB_0'~\cup~ayB_0''~\cup~x_1B_0'''~\cup
x_2B_2\big]~\cup  \\
&~x_2\big[A_2'~\cup~A_2''~\cup~x_1B_0^*~\cup~
y^{-1}B_2'~\cup~(ay)^{-1}B_2''~\cup~x_2B_2'''\big]~\cup  \\
&~x_3\big[x_1A^*_0~\cup~x_2A_2\big]  \\
=&~T_1~\cup~x_1\big[T_0 \setminus (y^{-1}B_2'~\cup~(ay)^{-1}B_2''\big]~\cup \\ 
&~x_2\big[A_2~\cup~y^{-1}B_2'~\cup~(ay)^{-1}B_2''~\cup~A_0^*\big]~\cup \\ 
&~x_3\big[B_0^*~\cup~B_2\big].
\end{align*}
Combining this with \eqref{eq:U'} we find that 
\begin{equation}\label{eq:nuU'}
\begin{aligned}
\nu(U')=&~
\big[x_1U'~\cup~y^{-1}A_2'~\cup~(ay)^{-1}A_2''~\cup~x_2y^{-1}B_2'~\cup~x_2(ay)^{-1}B_2''\big]~
\setminus~\big[x_1y^{-1}B_2'~\cup \\
&~x_1(ay)^{-1}B_2''~\cup~x_3y^{-1}A_2'~\cup~x_3(ay)^{-1}A_2''\big]  \\
=&~
x_1\big[U'~\cup~x_1y^{-1}A_2'~\cup~x_1(ay)^{-1}A_2''~\cup~x_3y^{-1}B_2'~\cup~x_3(ay)^{-1}B_2''\big]~
\setminus~ \\
&~x_1\big[y^{-1}B_2'~\cup~(ay)^{-1}B_2''~\cup~x_2y^{-1}A_2'~\cup~x_2(ay)^{-1}A_2''\big].  
\end{aligned}
\end{equation}

On the other hand, it follows from \eqref{eq:comp} that the subset 
$x_2y^{-1}A_2' \cup x_2(ay)^{-1}A_2'' \subseteq U'$  
covers the same components of $\G[U]$ as the set $x_1y^{-1}A_2' \cup x_1(ay)^{-1}A_2''$, and  
the subset $y^{-1}B_2'\cup (ay)^{-1}B_2'' \subseteq U'$  
covers the same components as the set $x_3y^{-1}B_2' \cup x_3(ay)^{-1}A_2''$.
Thus, letting $C=x_2y^{-1}A_2' \cup x_2(ay)^{-1}A_2'' \cup y^{-1}B_2' \cup~(ay)^{-1}B_2''$ and 
$C'=x_1y^{-1}A_2' \cup x_1(ay)^{-1}A_2'' \cup x_3y^{-1}B_2' \cup x_3(ay)^{-1}A_2''$, 
we have that $(U' \cup C') \setminus C$ is a transversal of $\G[U]$. 
As $\nu(U')=x_1((U' \cup C') \setminus C)$, see \eqref{eq:nuU'}, this proves that  
$\nu(U')$ is a transversal of $\G[G\setminus U]$, as required.
\end{proof}

In view of Lemmas~\ref{new-condition},~\ref{sigma's-1} and~\ref{sigma's-2}, 
the existence of a nowhere-zero $\Z_3$-flow in $\G_3$ follows from the following lemma.
\medskip

\begin{lem}\label{G3} 
Assuming the notations in case (ii) of Lemma~\ref{reduction}, let  
$\G=\G_3$, $Y=\sg{y}$ and $\Sigma=\cay(Y,X \setminus \{x\})$. 
Let $\H=\{a^i \mid i \in [0,p-1], 
i  \equiv 0\!\!\pmod 2\}$ and $\Y=\H \cup y(H \setminus \H) \cup y^2\H$.
Then the subgraphs $\Sigma_0=\Sigma [\Y]$ and $\Sigma_1=\Sigma[Y \setminus \Y]$ of $\G$ satisfy the conditions in Lemmas~\ref{sigma's-1} and \ref{sigma's-2}.
\end{lem}
\begin{proof}
Then $y^3=a^s$ for some $s \in [1,p-1]$.  Let $\varepsilon \in \{0,1\}$. 
Using \eqref{eq:edges}, one obtains that the $y$-edges of $\Sigma_\varepsilon$
is of the form $\{y^2a^{i},a^{i+s}\}$, where $i \equiv \varepsilon\!\!\pmod 2$ and 
$i < p-s$ if $s$ is even and $i > p-s$ if $s$ is odd. It follows from this 
that exactly one of $y^2$ and $y^2a^{p-1}$ is an incident with
a $y$-edge. Also, $\Sigma_0$ has two $a$-edges, namely
$\{1,a^{p-1}\}$ and $\{y^2,y^2a^{p-1}\}$; and $\{y,ya^{p-1}\}$ is 
the only $a$-edge of $\Sigma_1$. All these yield that $\Sigma_\varepsilon$ satisfies all  
conditions in Lemma~\ref{sigma's-1}. It is easily seen that 
there is a reduced transversal $T_\varepsilon$ of $\Sigma_\varepsilon$ for $\varepsilon\in [0,1]$, 
such that $T_\varepsilon \cap y^2H$ consists of the isolated vertices of $\Sigma_\varepsilon$ that are contained in $y^2H$, and 
\[
T_0 \cap H=H \setminus\{1\},~T_0 \cap yH=yH,~T_1 \cap H=H,~
T_1 \cap yH=yH \setminus\{y\}.
\]
These show that the conditions in (i) and (ii) of Lemma~\ref{sigma's-2} also hold.    
\end{proof}

We now turn to the graph $\G_4$. 
In view of Lemma~\ref{G4--1}, we may assume that $y^3=a^s$ for some 
$s \in [1,p-2]$. We distinguish two cases according to whether $s=1$ or not.

\begin{lem}\label{G4-1} 
Assuming the notations in case (ii) of Lemma~\ref{reduction}, let  
$\G=\G_4$, $Y=\sg{y}$ and $\Sigma=\cay(Y,X \setminus \{x\})$. 
Assume that $y^3=a$ and let  
\[
\H=\begin{cases} 
\big\{a^{i} \mid  i\in [0,p-1],~i \equiv 0,1\!\!\!\!\!\pmod 4\big\} & \text{if}~
p \equiv 1\!\!\!\!\!\pmod 4, \\ 
\big\{a^{p-1}, a^{i} \mid  i\in [0,p-3],~i \equiv 0,1\!\!\!\!\!\pmod 4\big\} & \text{if}~
p \equiv 3\!\!\!\!\!\pmod 4
\end{cases}
\]
and $\Y=\H \cup y(Y \setminus \H)  \cup y^2\H$. 
Then the subgraphs $\Sigma_0=\Sigma [\Y]$ and $\Sigma_1=\Sigma[Y \setminus \Y]$ of $\G_4$ satisfy the 
conditions in Lemmas~\ref{sigma's-1} and \ref{sigma's-2}.
\end{lem}
\begin{proof}
We distinguish two cases according to the reside of $p$ modulo $4$. 
\medskip

\noindent{\bf Case~1.} $p \equiv 1\!\!\pmod 4$.\ Using~\eqref{eq:edges}, we compute 
that the components of $\Sigma_0$ with at least $2$ vertices are the paths  
\[
(ya^{4i-1},y^2a^{4i},a^{4i+1},ya^{4i+2})\quad (i \in [1,(p-5)/4]),
\]  
and the tree made of the path $(ya^2,a,y^2a^{p-1},ya^{p-2})$ and the edges 
$\{a,y^2\}$ and $\{y^2a^{p-1},1\}$.
Clearly, there exists a reduced transversal $T_0$ of $\Sigma_0$ consisting of 
the isolated vertices of $\Sigma_0$. In particular, $T_0 \cap yH=\emptyset$ and 
$a^{p-1} \in T_0$. 

Then, for the components of $\Sigma_1$ with at least $2$ vertices, we find that these are the paths 
\[
(ya^{4i+1},y^2a^{4i+2},a^{4i+3},ya^{4i+4})\quad (i \in [0,(p-5)/4]).
\]  
There exists a reduced transversal $T_1$ of $\Sigma_1$ 
consisting of the isolated vertices of $\Sigma_1$, in particular, $T_1 \cap yH=\{y\}$. 
It follows that $\Sigma_0$ and $\Sigma_1$ satisfy the conditions in Lemma~\ref{sigma's-1} and also that the conditions in Lemma~\ref{sigma's-2} hold 
for $T_0$ and $T_1$.
\medskip

\noindent{\bf Case~2.} $p \equiv 3\!\!\pmod 4$.\ In this case, the components of 
$\Sigma_0$ with at least $2$ vertices are the paths  
\[
(ya^{4i-1},y^2a^{4i},a^{4i+1},ya^{4i+2})\quad (i \in [1,(p-7)/4]),
\]  
the path $(ya^{p-4},y^2a^{p-3},a^{p-1})$, and the tree 
made of the path $(ya^2,a,y^2a^{p-1},ya^{p-2})$ 
and the edges $\{a,y^2\}$ and $\{y^2a^{p-1},1\}$.
Again, there exists a reduced transversal $T_0$ of $\Sigma_0$ consisting of 
the isolated vertices of $\Sigma_0$. In particular, $T_0 \cap yH=\emptyset$ and 
$a^{p-1} \in T_0$.  

The components of $\Sigma_1$ with at least $2$ vertices are the paths
\[
(ya^{4i+1},y^2a^{4i+2},a^{4i+3},ya^{4i+4})\quad (i \in [0,(p-7)/4]),
\]  
and the path $(y^2a^{p-4},a^{p-2},ya^{p-1})$. 
There exists a reduced transversal $T_1$ of $\Sigma_1$ 
consisting of the isolated vertices of $\Sigma_1$ plus $a^{p-1}$, in particular, 
$T_1 \cap yH=\{y\}$.  It follows that $\Sigma_0$ and $\Sigma_1$ satisfy the conditions in Lemma~\ref{sigma's-1} and also that the conditions in Lemma~\ref{sigma's-2} hold 
for $T_0$ and $T_1$.
\end{proof}

\begin{table}[t!]
\begin{center}
\renewcommand{\arraystretch}{1.25}
{\small
\begin{tabular}{|c|c|c|} \hline
$r$   & $\varepsilon$  &  The edges of $\Sigma_\varepsilon$  \\ \hline 
even &  $0$   &  
\begin{tabular}{ll}
$a^{(2i+1)s-1} \sim ya^{(2i+1)s}$ & ($i \in [0,(r-2)/2]$)\\ 
$ya^{2is-1} \sim y^2a^{2is}$ & ($i \in [0,(r-2)/2]$) \\
$y^2a^{(2i+1)s-1} \sim a^{(2i+2)s}$ & ($i \in [0,(r-4)/2]$) \\
\end{tabular} \\    \hline  
even &  $1$   &  
\begin{tabular}{ll}  
$a^{2is-1} \sim ya^{2is}$ &  ($i \in [0,(r-2)/2]$) \\ 
$ya^{(2i+1)s-1} \sim y^2a^{(2i+1)s}$ &  ($i \in [0,(r-2)/2]$) \\
$y^2a^{2is-1} \sim a^{(2i+1)s}$ &  ($i \in [0,(r-2)/2]$) \\
$y^2a^{(r-1)s+i} \sim a^{rs+i+1}$ & ($i \in [0,t-2]$) \\
$y^2a^{(r-1)s+i} \sim a^{rs+i}$ & ($i \in [0,t-1]$) \\ 
\end{tabular} \\ \hline 
odd &  $0$  &  
\begin{tabular}{ll}
$a^{(2i+1)s-1} \sim ya^{(2i+1)s}$ & ($i \in [0,(r-5)/2]$)\\ 
$a^{(r-1)s-1} \sim ya^{(r-1)s}$ & \\ 
$ya^{2is-1} \sim y^2a^{2is}$ & ($i \in [0,(r-3)/2]$) \\
$y^2a^{(2i+1)s-1} \sim a^{(2i+2)s}$ & ($i \in [0,(r-5)/2]$) \\
$y^2a^{(r-3)s+i} \sim a^{(r-2)s+i}$ & ($i \in [0,s-1]$) \\
$y^2a^{(r-3)s+i} \sim a^{(r-2)s+i+1}$ & ($i \in [0,s-2]$) \\
\end{tabular} \\  \hline  
odd &  $1$     &  
\begin{tabular}{ll}
$a^{2is-1} \sim ya^{2is}$ & ($i \in [0,(r-3)/2]$)\\ 
$ya^{(2i+1)s-1} \sim y^2a^{(2i+1)s}$ & ($i \in [0,(r-5)/2]$) \\
$ya^{(r-1)s-1} \sim y^2a^{(r-1)s}$ & \\
$y^2a^{2is-1} \sim a^{(2i+1)s}$ & ($i \in [0,(r-5)/2]$) \\
$y^2a^{(r-1)s+i} \sim a^{rs+i}$ & ($i \in [0,t-1]$) \\
$y^2a^{(r-1)s+i} \sim a^{rs+i+1}$ & ($i \in [0,t-2]$) \\
\end{tabular}
\\ \hline 
\end{tabular}}
\end{center} 
\caption{The edges of $\Sigma_0$ and $\Sigma_1$, where $y^3=a^s$ and 
$2 \le s \le (p-1)/2$.}
\label{tab1}
\end{table}

\begin{lem}\label{G4-s} 
Assuming the notations in case (ii) of Lemma~\ref{reduction}, let  
$\G=\G_4$, $Y=\sg{y}$ and $\Sigma=\cay(Y,X \setminus \{x\})$. 
Assume that $y^3=a^s$ or $a^{-s}$ for some $s \in [2,(p-1)/2]$. 
Let $r, s$ be the unique positive integers such that $p=rs+t$ and $t < s$, and 
let  
\[
\H=\begin{cases} 
\big\{a^{2is}, \ldots, a^{2is+s-1} \mid i \in [0,(r-2)/2]\big\} & \text{if}~
r \equiv 0\!\!\!\!\!\pmod 2, \\ 
\big\{a^{2is}, \ldots, a^{2is+s-1} \mid i \in [0,(r-3)/2]\big\}~\cup~ 
\{ a^{(r-2)s},\ldots,a^{(r-1)s-1}\}  & \text{if}~r \equiv 1\!\!\!\!\!\pmod 2
\end{cases}
\]
and $\Y=\H \cup y(Y \setminus \H)  \cup y^2\H$. 
Then the subgraphs $\Sigma [\Y]$ and $\Sigma[Y \setminus \Y]$ of $\G_4$ satisfy the 
conditions in Lemmas~\ref{sigma's-1} and \ref{sigma's-2}.
\end{lem}
\begin{proof}
Let $\H_0=\H$ and $\H_1=H \setminus \H$. 
\medskip

\noindent{\bf Case~1.} $y^3=a^s$.\ We determined the edges of $\Sigma_0$ and $\Sigma_1$ by a straightforward computation, the results are recorded in Table~\ref{tab1}. One can quickly verify that $\Sigma_0$ and $\Sigma_1$ are pseudoforests which satisfy the conditions in Lemma~\ref{sigma's-1}. It remains to find suitable reduced transversals $T_0$ of $\Sigma_0$ and $T_1$ of $\Sigma_1$. 
If $r$ is even, then define the subsets $T_0 \subseteq \Y$ and $T_1 \subseteq Y \setminus \Y$ as 
\begin{align*}
T_0=& 
\H_0~\cup~
y \big[ \H_1 \setminus (\{a^{(2i+1)s} \mid i \in [0,(r-2)/2]\} \big]~\cup \\&~
y^2\big[ (\H_0 \setminus \{a^{2si}, a^{(2i+1)s-1} \mid  i \in [0,(r-2)/2]\})~\cup~ 
\{a^{(r-1)s-1}\}\big], \\
T_1=& 
(\H_1 \setminus \{a^s, a^{rs+i} \mid i \in [0,t-1]\})~\cup~
y\big[\H_0 \setminus \{a^{(r-1)s-1},a^{2is} \mid i \in [0,(r-2)/2]\} \big]~\cup \\ &~
y^2\big[\H_1 \setminus \big(\, \{a^{(2i+1)s}, a^{(2i+2)s-1} \mid i \in [0,(r-4)/2]\}~\cup~
\{ a^{(r-1)s+i+t}\mid i \in [0,s-1]\}\, \big) \big]; 
\end{align*}
and if $r$ is odd, then as 
\begin{align*}
T_0=&  
(\H_0 \setminus \{ a^{(r-2)s+i} \mid i \in [0,s-1]\})~\cup~
y\big[\H_1 \setminus  \{a^{(r-1)s}, a^{(2i+1)s} \mid i  \in [0,(r-5)/2]\} \big]~\cup \\ &~
y^2\big[ \H_0 \setminus (\{a^{2is}, a^{(2i+1)s-1} \mid i \in [0,(r-5)/2])\}~\cup~ 
\{a^{(r-3)s+i} \mid i\in [0,s-1]\}) \big],  \\
T_1=& 
(\H_1 \setminus \{a^s, a^{rs+i} \mid i \in [0,t-1]\})~\cup~
y\big[\H_0 \setminus \{a^{(r-1)s-1},a^{2is} \mid i \in [0,(r-3)/2]\} \big]~\cup \\ &~
y^2\big[ \H_1 \setminus (\{a^{(r-4)s}, a^{(2i+1)s}, a^{(2i+2)s-1} \mid i \in [0,(r-7)/2])\}~\cup~ 
\{a^{(r-1)s+i} \mid i\in [0,t-1]\}) \big]. 
\end{align*}

We claim that $T_0$ is a reduced transversal of $\Sigma_0$ and $T_1$ is a reduced transversal of 
$\Sigma_1$ and that the conditions in Lemma~\ref{sigma's-2} hold for $T_0$ and $T_1$. 
We give a proof only in the case where $r$ is even, as the argument in the case where $r$ is odd goes in the same way. 

It is easy to see that $T_0$ is a reduced transversal of $\Sigma_0$ and $T_1$ is a reduced 
transversal of $\Sigma_1$ and the conditions in (i) of Lemma~\ref{sigma's-2} hold. 
As $T_0 \cap H=\H_0$, we can choose $T_{0,1} \subseteq T_0 \cap H$ so that 
$yT_{0,1}=T_1 \cap yH$, in particular, the condition in (iii) with $\varepsilon=0$ holds.
To derive that the condition with $\varepsilon=1$ also holds, observe first that 
there exists a subsets $T_{1,1} \subseteq T_1 \cap \H_1$ such that 
\[
yT_{1,1}=
(T_0 \cap yH)\setminus \{ya^{rs},ya^{rs+1}\ldots,ya^{rs+t-1} \}.
\]
On the other hand, as $t < s$, 
\[
\{y^2a^{rs},y^2a^{rs+1}\ldots,y^2a^{rs+t-1}\} \subseteq 
\{y^2a^{(r-1)s+t},y^2a^{(r-1)s+t+1}\ldots,y^2a^{rs+t-1}\}=T_1 \cap y^2H.
\]
Thus, there exists a subset $T_{1,-1} \subseteq T_1 \cap y^2H$ such that 
$yT_{1,1}~\cup~y^{-1}T_{1,-1}=T_0 \cap yH$, i.\,e., the condition in (iii) 
with $\varepsilon=1$ also holds.
\medskip

\noindent{\bf Case~2.} $y^3=a^{-s}$.\ The edges of $\Sigma_0$ and $\Sigma_1$ having an end-vertex in $yH$ are the same as in the previous case, hence these are listed in Table~1. 

Let $r$ be even. Then the remaining edges of $\Sigma_0$ are 
$y^2a^{(2i+1)s-1} \sim a^{2is}$ ($i \in [0,(r-2)/2]$), and the remaining edges of 
$\Sigma_1$ are 
\[
\begin{tabular}{ll}
$y^2a^{(2i+2)s-1} \sim a^{(2i+1)s}$ & ($i \in [0,(r-4)/2]$), \\
$y^2a^{rs+i-1} \sim a^{(r-1)s+i}$ & ($i \in [0,t]$), \\
$y^2a^{rs+i} \sim a^{(r-1)s+i}$ & ($i \in [0,t-1]$). 
\end{tabular}
\]
We leave to the reader to check that the subset $T_0 \subseteq \Y$ defined  as in 
Case~1 for $r$ even and the subset $T_1 \subseteq Y \setminus \Y$ defined as  
\begin{align*}
T_1=& 
(\H_1 \setminus \{a^{(r-1)s+i} \mid i \in [0,t]\})~\cup~
y\big[\H_0 \setminus \{a^{(2is} \mid i \in [0,(r-2)/2]\} \big]~\cup \\ &~
y^2\big[ \H_1 \setminus \big(\, \{a^{(2i+1)s}, a^{(2i+2)s-1} \mid i \in [0,(r-4)/2]\}~\cup~
\{ a^{(r-1)s}, a^{rs+i}\mid i \in [0,t-1]\}\, \big) \big]; 
\end{align*}
are reduced transversals of $\Sigma_0$ and $\Sigma_1$, respectively, and that the conditions in Lemma~\ref{sigma's-2} hold for $T_0$ and $T_1$. 

Let $r$ be odd. Then the remaining edges of $\Sigma_0$ are 
\[
\begin{tabular}{ll}
$y^2a^{(2i+1)s-1} \sim a^{2is}$ & ($i \in [0,(r-5)/2]$), \\
$y^2a^{(r-2)s+i-1} \sim a^{(r-3)s+i}$ & ($i \in [0,s]$), \\
$y^2a^{(r-2)s+i} \sim a^{(r-3)s+i}$ & ($i \in [0,s-1]$),
\end{tabular}
\]
and the remaining edges of 
$\Sigma_1$ are 
\[
\begin{tabular}{ll}
$y^2a^{(2i+2)s-1} \sim a^{(2i+1)s}$ & ($i \in [0,(r-5)/2]$), \\
$y^2a^{rs+i-1} \sim a^{(r-1)s+i}$ & ($i \in [0,t]$), \\
$y^2a^{rs+i} \sim a^{(r-1)s+i}$ & ($i \in [0,t-1]$). 
\end{tabular}
\]
We leave to the reader to check that the subsets $T_0 \subseteq \Y$ 
and $T_1 \subseteq Y \setminus \Y$ defined  as 
\begin{align*}
T_0=&  
(\H_0 \setminus \{g a^{(r-3)s+i} \mid i \in [0,s]\})~\cup~
y\big[\H_1 \setminus  \{a^{(r-1)s}, a^{(2i+1)s} \mid i  \in [0,(r-5)/2]\} \big]~\cup \\ &~
y^2\big[ \H_0 \setminus (\{a^{2is}, a^{(2i+1)s-1} \mid i \in [0,(r-5)/2])\}~\cup~ 
\{a^{(r-3)s}, a^{(r-2)s+i} \mid i\in [0,s-1]\}) \big],  \\
T_1=& 
(\H_1 \setminus \{a^{(r-1)s+i} \mid i \in [0,t]\})~\cup~
y\big[\H_0 \setminus \{a^{(2is} \mid i \in [0,(r-3)/2]\} \big]~\cup \\ &~
y^2\big[ \H_1 \setminus (\{a^{(r-1)s}, a^{(2i+1)s}, a^{(2i+2)s-1} \mid i \in [0,(r-5)/2])\}~\cup~ 
\{a^{rs+i} \mid i\in [0,t-1]\}) \big]
\end{align*}
are reduced transversals of $\Sigma_0$ and $\Sigma_1$, respectively, and that the conditions in Lemma~\ref{sigma's-2} hold for $T_0$ and $T_1$. 
\end{proof}
%---------------------------------------------------------------------------------------------------%

\end{document}